\documentclass[12pt,reqno]{amsart}
\usepackage{amsmath,amssymb,amscd,verbatim,color}

\makeatletter
\@namedef{subjclassname@2020}{%
 \textup{2020} Mathematics Subject Classification}
\makeatother
\usepackage[colorlinks=true]{hyperref}
\usepackage{amsfonts}
\usepackage[top=35mm, bottom=35mm, left=30mm, right=30mm]{geometry}

\usepackage{xcolor}
\theoremstyle{plain}
\newtheorem{thm}{Theorem}[section]
\newtheorem{cor}[thm]{Corollary}

\newtheorem{lem}[thm]{Lemma}
\newtheorem{prop}[thm]{Proposition}

\theoremstyle{definition}
\newtheorem{defn}[thm]{Definition}
\newtheorem{exam}[thm]{Example}
\newtheorem{rem}[thm]{Remark}
\newtheorem{claim}{Claim}
\newtheorem{step}{Step}

\numberwithin{equation}{section}

\newcommand{\mb}{\mathbb}
\newcommand{\mc}{\mathcal}
\newcommand{\lk}{\left}
\newcommand{\re}{\right}

\allowdisplaybreaks

\begin{document}
	
	\title[Directional Kronecker algebra for $\mathbb{Z}^q$-actions]{Directional Kronecker algebra for $\mathbb{Z}^q$-actions}

	\author[C. Liu L. Xu]{Chunlin Liu and Leiye Xu}
	\address{C. Liu: CAS Wu Wen-Tsun Key Laboratory of Mathematics, University of Science and Technology of China, Hefei, Anhui, 230026, PR China}
	
	\email{lcl666@mail.ustc.edu.cn}
	
	\address{L. Xu: CAS Wu Wen-Tsun Key Laboratory of Mathematics, University of Science and Technology of China, Hefei, Anhui, 230026, PR China}

	\email{leoasa@mail.ustc.edu.cn}

	
	\begin{abstract}
   In this paper, directional sequence entropy and directional Kronecker algebra for $\mathbb{Z}^q$-systems are introduced. The relation between sequence entropy  and directional sequence entropy are established.  Meanwhile, direcitonal discrete spectrum systems and directional null systems are defined. It is shown that a $\mathbb{Z}^q$-system has directional discrete spectrum if and only if it is directional null. Moreover, it turns out that a $\mathbb{Z}^q$-system has directional discrete spectrum along $q$ linearly independent directions if and only if it has discrete spectrum.
   \end{abstract}

	\maketitle
	
	\section{Introduction}
	Given $q\in \mathbb{N}$. Throughout this paper by a $\mathbb{Z}^q$-measure preserving dynamical system ($\mathbb{Z}^q$-MPS for short) we mean a quadruple $(X,\mathcal{B}_X,\mu, T)$,
	where $X$ is a compact metric space, the $\mathbb{Z}^q$-action $T$ is a homeomorphism from the additive group $\mathbb{Z}^q$ to the group of homeomorphisms of $X$, $\mathcal{B}_X$ is the Borel $\sigma$-algbra of $X$ and $\mu$ is  an invariant Borel probability measure.  We denote the corresponding homeomorphism of $X$ by $T^{\vec{v}}$ for any $\vec{v}\in \mathbb{Z}^q$, so that $T^{\vec{v}}\circ T^{\vec{w}}=T^{\vec{v}+\vec{w}}$ and $T^{\vec{0}} $ is the identity on $X$.

 Kolmogorov \cite{K} introduced the notion of entropy to dynamical systems, which played an important role in the study of ergodic theory and topological dynamics to measure the chaoticity or unpredictability of a system. To investigate the complexity of entropy zero systems, many invariants were introduced. Ku\v shnirenko \cite{Kus} developed an invariant, sequence entropy, to distinguish  non-isomorphic entropy zero systems. We refer to \cite{P,Sa} for further study on sequence entropy. Katok and Thouvenot \cite{KT} introduced the notion of slow entropy for measure preserving actions of discrete amenable groups to measure the complexity of entropy zero systems. Recently, Dou, Huang and Park \cite{DHP} introduced the notion of entropy dimension to classify the growth rate of the orbits of entropy zero systems. To investigate the Cellular Automaton map together with the Bernoulli shift, Milnor \cite{Mil} defined directional entropy. Many of its properties were further studied in \cite{Ak,Bro,CK,P1,Pa,ZY}. 

In this paper, we introduce a new invariant, directional sequence entropy, for $\mathbb{Z}^q$-actions, based on classical sequence entropy and directional entropy. We investigate some properties of directional sequence entropy, analogously to classical sequence entropy. Moreover, we discuss the relation between the spectrum and directional sequence entropy. One of our main results is to show the directional version of Ku\v shnirenko theorem, which tells us that a $\mathbb{Z}^q$-MPS has directional discrete spectrum if and only if its directional sequence entropy is always zero for any sequence along corresponding direction. In the proof of this result, our main difficulty is that many theories of classical sequence entropy can not be  directly applied to our proof such as spectral mixing theorem of Koopman-von Neumann \cite{KN}, since these results  are only available for the case of group actions. Motivated by a $\mb{Z}$-MPS, which was introduced by Park \cite{Pa}, we overcome this difficulty by establishing the relaiton between directional Kronecker algera and classical Kronecker algebra. Taking advantage of the relation between the two, we can show our results by some classical consequences for the case of $\mb{Z}$-actions. Meanwhile, we show that directional sequence entropy is closely relavant to classical sequence entropy and a $\mathbb{Z}^q$-MPS has directional discrete spectrum along $q$ linearly independent directions if and only if it has discrete spectrum.

For convenience, we focus on $\mathbb{Z}^2$-MPS. Using exactly the same methods, we can obtain corresponding results of $\mathbb{Z}^q$-MPS for any integer $q\geq 2$, which are introduced in Appendix B.
 To be precise, let $(X ,\mathcal{B}_X, \mu, T)$ be a $\mathbb{Z}^2$-MPS and $\vec{v}=(1,\beta)\in \mb{R}^2$ be a direction vector (For simplicity, we write $\vec{v}$ as $(1,\beta)$. In fact all results in this paper are true when $\vec{v}=(0,1)$ since this case is the same as the case of $\mathbb{Z}$-actions.). 
For a finite measurable partition $\alpha$ of $X$, let $$H_{\mu}(\alpha)=-\sum_{A\in \alpha}\mu(A)\log{\mu(A)}.$$ We put $$\Lambda^{\vec{v}}(b)=\left\{(m,n)\in\mathbb{Z}^2:\beta m-b/2\leq n\leq \beta m+b/2\right\}.$$
Given a finite measurable partition $\alpha$ of $X$. 
For any infinite subset $S=\{(m_i,n_i)\}_{i=1}^{\infty}$ of $\Lambda^{\vec{v}}(b)$ that $\{m_i\}_{i=1}^{\infty}$ is strictly monotone, we put
$$ h^S_{\mu}(T,\alpha)= \limsup_{k\to \infty} \frac{1}{k}H_\mu \lk(\bigvee_{i=1}^k T^{-(m_i,n_i)} \alpha\re).$$
Then we can define the directional sequence entroy of $T$ for the infinite subset $S$ by
$$h^S_{\mu}(T)=\sup_{\alpha}h^S_{\mu}(T,\alpha),$$
where the supremum is taken over all finite measurable partitions of $X$.  

 Define the $\vec{v}$-directional Kronecker algebra by
   $$\mathcal{K}_\mu^{\vec{v}}(b)=\left\{B\in\mathcal{B}_X: \overline{\lk\{U_T^{(m,n)}1_B :(m,n)\in \Lambda^{\vec{v}}(b) \re\}}\text{ is compact in } L^2(X,\mathcal{B}_X,\mu) \right\},$$
   where $U_T^{(m,n)}:L^2(X,\mathcal{B}_X,\mu)\to L^2(X,\mathcal{B}_X,\mu)$ is the unitary operator such that
   $$U_T^{(m,n)}f=f\circ T^{(m,n)}\text{ for all }f\in L^2(X,\mathcal{B}_X,\mu)$$ and $1_B$ is the indicator function of $B\in \mathcal{B}_X$, that is, 
      $$ 1_B(x)=\left\{
     \begin{array}{lr}
      1,                 & x\in B,\\
      0,                  & x\notin B.\\
      \end{array}
     \right. $$
   In Section 3, we prove that $\mathcal{K}_\mu^{\vec{v}}(b)$ is a $\sigma$-algebra and the definition of $\mathcal{K}_\mu^{\vec{v}}(b)$ is independent of the selection of $b\in (0,\infty)$. So we omit $b$ in $\mathcal{K}_\mu^{\vec{v}}(b)$ and write it as $\mathcal{K}_\mu^{\vec{v}}$.

Motivated by the work in \cite{HMY}, we can show the following consequence.
\begin{thm} \label{thm-3}
	Let $(X,\mathcal{B}_X,\mu,T)$ be a  $\mathbb{Z}^2$-MPS,  $\vec{v}=(1,\beta)\in\mathbb{R}^2$ be a direction vector and $b\in (0,\infty)$. Given a finite measurable partition $\alpha$ of $X$. For any infinite subset $S'$ of $\Lambda^{\vec{v}}(b)$,
	$$h^{S'}_{\mu}(T,\alpha)\leq H_\mu(\alpha|\mathcal{K}_\mu^{\vec{v}}).$$
    Moreover, there exists an infinite subset $S=\{(m_i,n_i)\}_{i=1}^{\infty}$ of $ \Lambda^{\vec{v}}(b)$ such that $\{m_i\}_{i=1}^{\infty}$ is strictly monotone and $$h^{S}_{\mu}(T,\alpha)=H_\mu(\alpha|\mathcal{K}_\mu^{\vec{v}}).$$
\end{thm}
\begin{rem}\label{r1}
	Let $\alpha$ be a finite measurable partition of $X$. We
	define 
	\begin{align}\label{4}h_\mu^{\vec{v},*}(\alpha)=\sup\limits_{S\subset \Lambda^{\vec{v}}(b)}\{h^{S}_{\mu}(T,\alpha)\},
	\end{align}
	where the supremum is taken over all infinite subsets $S=\{(m_i,n_i)\}_{i=1}^{\infty}$ of $ \Lambda^{\vec{v}}(b)$. By Theorem \ref{thm-3}, $h_\mu^{\vec{v},*}(\alpha)=H_\mu(\alpha|\mathcal{K}_\mu^{\vec{v}})$ and then
	\eqref{4} is well defined since it is independent of the selection of $b$.
\end{rem}
 In order to investigate  entropy zero systems, Ku\v shnirenko \cite{Kus} introduced the notion of sequence entropy and studied spectral theory via sequence entropy. He obtained a classical consequence, that is, a $\mathbb{Z}$-MPS is a discrete spectrum system if and only if it is a null system.
 Motivated by his work, we define directional discrete spectrum systems and directional null systems as follows.
\begin{defn}
	Let $(X,\mathcal{B}_X,\mu,T)$ be a $\mathbb{Z}^2$-MPS and $\vec{v}=(1,\beta)\in\mathbb{R}^2$ be a direction vector.
	\begin{itemize}
		\item[(a)]We say $\mu$ has $\vec{v}$-discrete spectrum if $\mathcal{K}_\mu^{\vec{v}}=\mathcal{B}_X$.
		\item[(b)] We say $(X,\mathcal{B}_X,\mu,T)$ is  $\vec{v}$-null if for any $b\in (0,\infty)$ and infinite subset $S=\{(m_i,n_i)\}_{i=1}^{\infty}$ of $ \Lambda^{\vec{v}}(b)$, $h^{S}_{\mu}(T)=0$.
	\end{itemize}
\end{defn}
We obtain one of our main results which is a directional version of Ku\v shnirenko theorem.
\begin{thm}\label{thm2}
	Let $(X,\mathcal{B}_X,\mu,T)$ be a $\mathbb{Z}^2$-MPS and $\vec{v}=(1,\beta)\in\mathbb{R}^2$ be a direction vector.  Then the following two conditions are equivalent.
	\begin{itemize}
		\item[(a)] $\mu$ has $\vec{v}$-discrete spectrum system.
		\item[(b)]$(X,\mathcal{B}_X,\mu,T)$ is $\vec{v}$-null.
	\end{itemize}
\end{thm}

Meanwhile, we establish the relation between sequence entropy (respectively discrete spectrum) and directional sequence entropy (respectively direcitonal  discrete spectrum) as follows.
\begin{thm}\label{thm3}
	Let $(X,\mathcal{B}_X,\mu,T)$ be a $\mathbb{Z}^2$-MPS. Given a finite measurable partition $\alpha$ of $X$. If there exists an infinite subset $S$ of $\mathbb{Z}^2$ such that $h^S_{\mu}(T,\alpha)>0$, then there is at most one direction $\vec{v}=(1,\beta)\in\mb{R}^2$ such that $h_\mu^{\vec{v},*}(\alpha)=0$.
\end{thm}
\begin{thm}\label{thm4}	
		Let $(X,\mathcal{B}_X,\mu,T)$ be a $\mathbb{Z}^2$-MPS. Then the following statements are equivalent.
		\begin{itemize}
			\item[(a)] There exist two directions $\vec{v}=(1,\beta_1)$, $\vec{w}=(1,\beta_2)\in\mb{R}^2$ with $\beta_1\neq \beta_2$ such that $\mu$ has $\vec{v}$-discrete spectrum and $\vec{w}$-discrete spectrum.
			\item[(b)]$\mu$ has discrete spectrum.
		\end{itemize}
\end{thm}
     This paper is organized as follows. In Section 2, we recall some basic notions that we use in this paper. In Section 3, we introduce the directional Kronecker algebra and study many properties of directional sequence entropy to describe it. In Section 4, we prove Theorem \ref{thm-3} and Theorem \ref{thm2}. In Section 5, we prove Theorem \ref{thm3} and Theorem \ref{thm4}. In Appendix A, we prove some consequences that are not directly relevant to the main purpose of this paper. In Appendix B, we introduce the corresponding results of $\mathbb{Z}^q$-MPS.

 \section{Preliminaries}
In this section we recall some notions of measure preserving systems that are used later (see \cite{DT,DR,EW,Peter}).
\subsection{General notions.}
In this article, the sets of real numbers, rational numbers, integers, non-negative integers and natural numbers are denoted by $\mb{R}$, $\mathbb{Q}$, $\mathbb{Z}$, $\mathbb{Z}_+$ and $\mathbb{N}$, respectively.  We use $\|f\|_p$ to denote the $L^p$-norm of a function $f$ defined in a Borel probability measure space $(X,\mathcal{B}_X,\mu)$, that is,  
  $$ \|f\|_p=  
  (\int_{X}|f|^pd\mu)^{\frac{1}{p}}, \text{ if }  1\leq p<\infty$$ and
 $$ \|f\|_{\infty}=\inf\{a\geq 0:\mu\lk(\{x:|f(x)|>a\}\re)=0\}.$$

\subsection{Conditional expectation.}
Let $\mathcal{X}$ be a sub-$\sigma$-algebra of $\mathcal{B}_X$. Then $L^2(X,\mathcal{X},\mu)$ is contained as a subspace in $L^2(X,\mathcal{B}_X,\mu)$  and the conditional expectation is given as the orthogonal projection map:
$$\mathbb{E}(\cdot|\mathcal{X}): L^2(X,\mathcal{B}_X,\mu)\to L^2(X,\mathcal{X},\mu).$$ 
It is known that the conditional expectation is uniquely determined by the following statements (see \cite{DR}).
\begin{itemize}
\item[(a)] $\mathbb{E}(f|\mathcal{X})$ is $\mc{X}$-measurable for all $f\in L^2(X,\mathcal{B}_X,\mu)$.
\item[(b)]For all $A\in \mc{X}$, $\int_A f d\mu=\int_A \mathbb{E}(f|\mathcal{X})d\mu. $
\end{itemize}

Therefore we have the following result of conditional expectation.
\begin{prop}\label{prop1}
Let $(X,\mathcal{B}_X,\mu)$ and $(Y,\mathcal{B}_Y,\nu)$ be two Borel probability spaces. Let $\mathcal{X}$ and $\mathcal{Y}$ be sub-$\sigma$-algebras of $\mathcal{B}_X$ and $\mathcal{B}_Y$, respectively. Then for any $f\in L^2(X,\mathcal{B}_X,\mu)$ and $g\in L^2(Y,\mathcal{B}_Y,\nu)$, 
$$\mathbb{E}(fg|\mathcal{X}\times \mathcal{Y})=\mathbb{E}(f|\mathcal{X})\mathbb{E}(g|\mathcal{Y})$$
for $\mu\times \nu$-a.e. $(x,y)\in X\times Y$,
where $\mathcal{X}\times \mathcal{Y}$ be the product $\sigma$-algebra of $\mathcal{X}$ and $\mathcal{Y}$.
\begin{proof}
For the completeness of the paper, we conclude the proof which is easy to obtain by above result in \cite{DR} and Fubini's theorem.
By the uniqueness of the conditional expectation, it is sufficient to prove that for any $A\times B\in \mc{X}\times \mc{Y}$,
\begin{equation}\label{5}
\int_{A\times B} f(x)g(y) d\mu(x)d\nu(y)=\int_{A\times B} \mathbb{E}(f|\mathcal{X})(x)\mathbb{E}(g|\mathcal{Y})(y)d\mu(x)d\nu(y).
\end{equation} 
By Fubini's theorem, we obtain
$$\int_{A\times B} f(x)g(y) d\mu(x)d\nu(y)=\int_{A} f(x) d\mu(x)\int_{B} g(y)d\nu(y)$$
and
$$\int_{A\times B} \mathbb{E}(f|\mathcal{X})(x)\mathbb{E}(g|\mathcal{Y})(y)d\mu(x)d\nu(y)=\int_{A} \mathbb{E}(f|\mathcal{X})(x)d\mu(x)\int_{B}\mathbb{E}(g|\mathcal{Y})(y)d\nu(y).$$
Then we immediately get \eqref{5} by (b) in this subsection.
\end{proof}
\end{prop}

\subsection{Classical Kronecker algebra and discrete spectrum.}
Let $(X,\mathcal{B}_X,\mu,T)$ be a $\mathbb{Z}$-MPS and $\mathcal{H}=L^2(X,\mathcal{B}_X,\mu)$. In complex Hilbert space $\mathcal{H}$, we define a unitary operator $U_T:\mathcal{H}\rightarrow \mathcal{H}$ by $$U_Tf=f\circ T$$  for all  $f\in\mathcal{H}.$  We say that $f$ is an almost periodic function if $\overline{\left\{U_T^n f:n\in \mathbb{Z}\right\}}$ is a compact subset of $\mathcal{H}$.
It is well known that the set of all bounded almost periodic functions forms a $U_T$-invariant and conjugation-invariant subalgebra of $\mathcal{H}$ (denoted by $\mathcal{A}_c$). The set of almost periodic functions is just the closure of $\mathcal{A}_c$ (denoted by $\mathcal{H}_c$). It is known that (see \cite[Theorem 1.2]{Zi}) there exists a $T$-invariant sub-$\sigma$-algebra $\mathcal{K}_{\mu}$ of $\mathcal{B}_X$ such that $$\mathcal{H}_c=L^2(X,\mathcal{K}_{\mu},\mu).$$ The sub-$\sigma$-algebra $\mathcal{K}_{\mu}$ is called the Kronecker algebra of $(X,\mathcal{B}_X,\mu,T)$.    It is easy to know that $\mathcal{K}_\mu$ consists of all $B\in \mathcal{B}_X$ such that
$$\overline{\{ U_T^n1_B:n\in \mathbb{Z} \}} \text{ is compact in }\mathcal{H}.$$
We say $\mu$ has discrete spectrum if $\mathcal{B}_X=\mathcal{K}_{\mu}$. Similarly, the Kronecker algebra and discrete spectrum for the case of $\mathbb{Z}^q$-actions can be defined (see \cite{G}).

Let $S\subset \mathbb{Z}_+$ be an infinite sequence. We define the upper density and lower density by
$$\overline{d}(S)=\limsup\limits_{n\to \infty}\frac{1}{n}|S\cap\{0,1,\ldots, n-1\}|$$
and
$$\underline{d}(S)=\liminf\limits_{n\to \infty}\frac{1}{n}|S\cap\{0,1,\ldots, n-1\}|$$
where $|A|$ is the number of elements of a finite set $A$. 
If $\overline{d}(S)=\underline{d}(S)=d$, we say $S$ has the density $d$.

  We recall the spectral mixing theorem of  Koopman-von Neumann \cite{KN}.

\begin{thm}\label{knt}
The Hilbert space  $\mathcal{H}$ can be decomposed as $\mathcal{H}=\mathcal{H}_c\oplus \mathcal{H}_c^{\perp}$, where
	$$\mathcal{H}_c=\overline{span\{f\in \mathcal{H}:\text{there exists } \lambda\in \mathbb{C},\text{ such that }  U_T(f)=\lambda f\}}$$
	and
	$$\mathcal{H}_c^{\perp}=\{f\in \mathcal{H}:\exists S\subset \mathbb{N},\text{ }d(S)=1,\text{ such that } \text{for all } g\in \mathcal{H},\lim\limits_{n\to \infty\atop n\in S}\langle U_T^nf,g\rangle=0\}$$
	where $d(S)$ is the density of $S$ and $\langle \cdot,\cdot\rangle: \mathcal{H}\times \mathcal{H}\to \mathbb{C},  (f,g)\mapsto \langle f,g\rangle=\int_X f(x)\overline{g(x)}dx  $ is an inner product on $\mathcal{H}$.
\end{thm}

\subsection{Sequence entropy}
Let $(X,\mathcal{B}_X,\mu,T)$ be a $\mathbb{Z}^2$-MPS.
Given a finite measurable partition $\alpha$ of $X$, let $$H_{\mu}(\alpha)=-\sum_{A\in \alpha}\mu(A)\log{\mu(A)}.$$ For any infinite subset $S=\{(m_i,n_i)\}_{i=1}^{\infty}$ of $\mathbb{Z}^2$, we put
$$ h^S_{\mu}(T,\alpha)= \limsup_{k\to \infty} \frac{1}{k}H_\mu \lk(\bigvee_{i=1}^k T^{-(m_i,n_i)} \alpha\re).$$
Then we define the sequence entroy of $T$ for the infinite subset $S$ by
$$h^S_{\mu}(T)=\sup_{\alpha}h^S_{\mu}(T,\alpha),$$
where the supremum is taken over all finite measurable partitions of $X$.


\subsection{$\epsilon$-net}
Let $A$ be a subset of a normed linear space $(X,\|\cdot\|)$ and $\epsilon>0$. A subset $A_{\epsilon}\subset X$ is called an $\epsilon$-net for $A$ if for each $x\in A$ there is an element of $A$ is within an $\epsilon$ distance to some element of $A_{\epsilon}$.

It is known that (see \cite[Page 28]{A}) a subset $A$ of a normed linear space $(X,\|\cdot\|)$ is compact if and only if for any $\epsilon>0$ there is a finite $\epsilon$-net $F_{\epsilon}\subset X$ for $A$. That is, there is a finite set $F_{\epsilon}\subset X$ such that 
$$A\subset \bigcup_{x\in F_{\epsilon}}B(x,\epsilon).$$

Note that the $\epsilon$-net mentioned in this paper is relative to $L^2$-norm unless we explicitly indicate otherwise.

\section{Directional Kronecker algebra}
 In this section we introduce the directional Kronecker algebra and describe it via directional sequence entropy.
Let $(X ,\mathcal{B}_X, \mu, T)$ be a $\mathbb{Z}^2$-MPS, $\vec{v}=(1,\beta)\in\mathbb{R}^2$ be a  direction vector and $b\in (0,\infty)$. We put $$\Lambda^{\vec{v}}(b)=\left\{(m,n)\in\mathbb{Z}^2:\beta m-b/2\leq n\leq \beta m+b/2\right\}.$$ Let $\mathcal{A}_c^{\vec{v}}(b)$ be the collection of $f\in L^2(X,\mathcal{B}_X,\mu)$ such that
$$\overline{\left\{U_T^{(m,n)}f:(m,n)\in \Lambda^{\vec{v}}(b) \right\}}$$ is compact in $L^2(X,\mathcal{B}_X,\mu).$
It is easy to see that  $\mathcal{A}_c^{\vec{v}}(b)$ is a $U_{T^{\vec{w}}}$-invariant for all $\vec{w}$ in $\mathbb{Z}^2$ and conjugation-invariant subalgebra of $L^2(X,\mathcal{B}_X,\mu)$.  It is known that (see \cite[Theorem 1.2]{Zi}) there exists a $T$-invariant sub-$\sigma$-algebra  $\mathcal{K}_\mu^{\vec{v}}(b)$  of $\mathcal{B}_X$ such that
\begin{align}\label{1}
\mathcal{A}_c^{\vec{v}}(b)=L^2(X,\mathcal{K}_\mu^{\vec{v}}(b),\mu).
\end{align}

 By \eqref{1}, the $\vec{v}$-directional Kronecker algebra of $(X ,\mathcal{B}_X, \mu, T)$ can be defined by
$$\mathcal{K}_\mu^{\vec{v}}(b)=\left\{B\in\mathcal{B}_X: \overline{\lk\{U_T^{(m,n)}1_B :(m,n)\in \Lambda^{\vec{v}}(b) \re\}}\text{ is compact in } L^2(X,\mathcal{B}_X,\mu) \right\}.$$

The following result shows that the definition of $\mathcal{K}_\mu^{\vec{v}}(b)$ is independent of the selection of $b\in (0,\infty)$.
    \begin{prop}\label{prop-1} Let $(X,\mathcal{B}_X,\mu,T)$ be a $\mathbb{Z}^2$-MPS and $\vec{v}=(1,\beta)\in\mathbb{R}^2$ be a direction vector.
 Then $\mathcal{K}_\mu^{\vec{v}}(b_1)=\mathcal{K}_\mu^{\vec{v}}(b_2)$ for any $b_1,b_2\in (0,\infty)$.
    \end{prop}
\begin{proof}We put $$U(B,b)=\overline{\{U_T^{(m,n)}1_B :(m,n)\in \Lambda^{\vec{v}}(b) \}}$$ for $b\in(0,\infty)$ and $B\in\mathcal{B}_X$.
Fix $b_1,b_2\in(0,\infty)$. Without loss of generality, we assume that $b_1>b_2$.

We first prove that $\mathcal{K}_\mu^{\vec{v}}(b_1)\subset \mathcal{K}_\mu^{\vec{v}}(b_2).$ For a fixed $B\in \mathcal{K}_\mu^{\vec{v}}(b_1)$, $U(B,b_1)$ is compact  in $L^2(X,\mathcal{B}_X,\mu)$ Hence $U(B,b_2)$  is compact  in $L^2(X,\mathcal{B}_X,\mu)$ since  $U(B,b_2)$ is closed in $L^2(X,\mathcal{B}_X,\mu)$ and  $U(B,b_2)\subset U(B,b_1).$
 It follows that $$\mathcal{K}_\mu^{\vec{v}}(b_1)\subset \mathcal{K}_\mu^{\vec{v}}(b_2).$$

Next we show that $\mathcal{K}_\mu^{\vec{v}}(b_2)\subset\mathcal{K}_\mu^{\vec{v}}(b_1).$  Given $B\in \mathcal{K}_\mu^{\vec{v}}(b_2),$    $U(B,b_2)$ is compact in $L^2(X,\mathcal{B}_X,\mu). $ Note that we can find a  finite subset $C$ of $\mathbb{Z}^2$ such that
  $$U(B,b_1)\subset \bigcup_{(m',n')\in C}\overline{\lk\{U_T^{(m+m',n+n')}1_B :(m,n)\in \Lambda^{\vec{v}}(b_2) \re\}}.$$
 It follows from the fact that $\overline{\lk\{U_T^{(m,n)}1_B :(m,n)\in \Lambda^{\vec{v}}(b_2) \re\}}$ is  compact in $L^2(X,\mathcal{B}_X,\mu)$ that $$\overline{\lk\{U_T^{(m+m',n+n')}1_B :(m,n)\in \Lambda^{\vec{v}}(b_2) \re\}}$$ is compact in $L^2(X,\mathcal{B}_X,\mu)$ for all $(m',n')\in C$. So $U(B,b_1)$ is compact in $L^2(X,\mathcal{B}_X,\mu)$, since it is closed in $L^2(X,\mathcal{B}_X,\mu)$. This implies that $B\in \mathcal{K}_\mu^{\vec{v}}(b_1).$ Therefore $$\mathcal{K}_\mu^{\vec{v}}(b_2)\subset\mathcal{K}_\mu^{\vec{v}}(b_1).$$
This ends the proof of Proposition \ref{prop-1}.
\end{proof}

    Now we establish the relation between direcitonal Kronecker algebra and directional sequence entropy. Let us begin with some notations. Since complex Hilbert space  $\mathcal{H}=L^2(X,\mathcal{B}_X,\mu)$ is separable, there exists an infinite subset 
    \begin{equation}\label{S1}
      S=\{(m_i,n_i)\}_{i=1}^{\infty} 
    \end{equation}
    of $\Lambda^{\vec{v}}(b)$
   such that $\lim\limits_{i\to \infty}\langle g,U_T^{(m_i,n_i)}f\rangle$ exists for any $f,g \in \mathcal{H}$. Now for a fixed $f\in \mathcal{H}$, we define $J:\mathcal{H}\to \mathbb{C}$ by $$J(g)=\lim\limits_{i\to \infty}\langle g,U_T^{(m_i,n_i)}f\rangle.$$ Obviously, $J$ is a continuous linear functional on $\mathcal{H}$. By Riesz representation theorem, there exists $S(f)\in \mathcal{H}$ such that
    \begin{equation}\label{S2}
     J(g)=\langle g,S(f)\rangle 
    \end{equation}
   
     for all $g\in \mathcal{H}$.

    \begin{lem}\label{lem-6}
    	Let $(X,\mathcal{B}_X,\mu,T)$ be a $\mathbb{Z}^2$-MPS, $\vec{v}=(1,\beta)\in\mathbb{R}^2$ be a direction vector and $b\in (0,\infty)$. Then for any   finite measurable partition $\alpha$ of $X$, there exists an infinite subset $S'=\{(m'_i,n'_i)\}_{i=1}^{\infty}$ of $ S$ such that $\{m'_i\}_{i=1}^{\infty}$ is strictly monotone and
    	$$h^{S'}_{\mu}(T)\geq \sum_{B\in\alpha}\int_X-S(1_B)\log{S(1_B)}d\mu,$$
    	where $S$ and $S(\cdot)$ are defined by \eqref{S1} and \eqref{S2}, respectively.
    \end{lem}
    \begin{proof}
    	Enumerate $\alpha$ as $\{B_1,\dots,B_k\}$.  Fix a finite measurable partition $\eta=\{D_1,\dots D_l\}$ of $X$. Let $S=\{(m_i,n_i)\}_{i=1}^{\infty} $ be define in \eqref{S1}. Then  
    	\begin{equation}\label{3.4}
    	\lim\limits_{i\to \infty}\langle U_T^{(m_i,n_i)}1_{B_s},1_{D_t}\rangle=\langle1_{D_t},S(1_{B_s})\rangle
    	\end{equation}
    	for any $s\in\{1,2,\ldots ,k\}$ and $t\in\{1,2,\ldots ,l\}$.
    	Therefore we deduce
    	\begin{align}\label{6}
    	\liminf_{i\to \infty}H_{\mu}(T^{-(m_i,n_i)}\alpha|\eta)\notag&=\liminf_{i\to \infty}\sum_{s,t}-\mu(T^{-(m_i,n_i)}B_s\cap D_t)\log{\lk(\frac{\mu(T^{-(m_i,n_i)}B_s\cap D_t)}{\mu(D_t)}\re)}\\
    	\notag&=\lim\limits_{i\to \infty}\sum_{s,t}-\langle U_T^{(m_i,n_i)}1_{B_s},1_{D_t}\rangle\log{\lk(\frac{\langle U_T^{(m_i,n_i)}1_{B_s},1_{D_t}\rangle}{\mu(D_t)}\re)}\\
    	\notag&\overset{\eqref{3.4}}=\sum_{s,t}-\langle S(1_{B_s}),1_{D_t}\rangle\log{\lk(\frac{\langle S(1_{B_s}),1_{D_t}\rangle}{\mu(D_t)}\re)}\\
    	&=\sum_{s,t}-\mu(D_t)\int_{D_t}\frac{S(1_{B_s})}{\mu(D_t)}d\mu\log{\lk(\int_{D_t}\frac{S(1_{B_s})}{\mu(D_t)}d\mu\re)}\\
    	\notag&\geq \sum_{s,t}-\mu(D_t)\int_{D_t}\frac{S(1_{B_s})}{\mu(D_t)}\log{(S(1_{B_s}))}d\mu\\
    	\notag&= \sum_s-\int_XS(1_{B_s})\log{\lk(S(1_{B_s})\re)}d\mu\\
    	\notag&=\sum_{B\in \alpha}-\int_XS(1_B)\log{\lk(S(1_B)\re)}d\mu,
    \end{align}
    	where in the inequality we use concavity of the function $-x\log{x}$.

    Now we can define inductively an infinite subset $S'=\{(m'_i,n'_i)\}_{i=1}^{\infty}$ of $ S$  by  using \eqref{6} repeatedly such that  $\{m'_i\}_{i=1}^{\infty}$ is strictly monotone and for any $i\in \mathbb{N}$, 
    	$$H_\mu\lk(T^{-(m'_i,n'_i)}\alpha|\bigvee_{j=1}^{i-1} T^{-(m'_j,n'_j)}\alpha\re)\geq \sum_{B\in \alpha}-\int_XS(1_B)\log{(S(1_B))}d\mu-\frac{1}{2^i}.$$
       	
     As $H_{\mu}\big(\bigvee_{i=1}^kT^{-(m_i,n_i)}\alpha\big)=H_{\mu}\big(\bigvee_{i=1}^{k-1}T^{-(m_i,n_i)}\alpha\big)+H_{\mu}\big(T^{-(m_k,n_k)}\alpha|\bigvee_{i=1}^{k-1}T^{-(m_i,n_i)}\alpha\big)$,  
     therefore
    	\begin{equation*}
    	\begin{split}
    h^{S'}_{\mu}(T)
    &=\limsup\limits_{k\to \infty}{\frac{1}{k}\sum_{i=1}^{k}H_\mu\lk(T^{-(m'_i,n'_i)}\alpha|\bigvee_{j=1}^{i-1} T^{-(m'_j,n'_j)}\alpha\re)}\\
    &\geq \limsup\limits_{k\to \infty}{\frac{1}{k}\sum_{i=1}^{k}\lk(\sum_{B\in \alpha}-\int_XS(1_B)\log{(S(1_B))}d\mu-\frac{1}{2^i}\re)}\\
    &=\sum_{B\in \alpha}-\int_XS(1_B)\log{(S(1_B))}d\mu.
    	\end{split}
    	\end{equation*}
   Now we finish the proof of Lemma \ref{lem-6}.
    \end{proof}

    For further discussion, we need a classcial consequence (\cite[Page 94]{Peter}).
    \begin{lem} \label{lem-2}Let $(X,\mathcal{B}_X,\mu)$ be a Borel probability space and $r\geq 1$ be a fixed integer. For each $\epsilon >0$, there exists  $\delta=\delta(\epsilon,r)>0$ such that if  $\alpha=\{A_1,A_2,\ldots,A_r\}$ and $\eta=\{B_1,B_2,\ldots,B_r\}$ are any two finite measurable partitions of $(X,\mathcal{B}_X,\mu)$ with $\sum_{j=1}^r\mu(A_j\Delta B_j)< \delta$ then $H_\mu(\alpha|\eta)+H_\mu(\eta|\alpha)<\epsilon$.
    \end{lem}
    \medskip

   By Lemma \ref{lem-6} and Lemma \ref{lem-2}, we can get the following result.
   \begin{thm}\label{thm-11}
   	Let $(X,\mathcal{B}_X,\mu,T)$ be a $\mathbb{Z}^2$-MPS, $\vec{v}=(1,\beta)\in\mathbb{R}^2$ be a direction vector and $b\in (0,\infty)$. Let $B\in \mathcal{B}_X$ and $P$ be an infinite subset of $\Lambda^{\vec{v}}(b)$. Then the following conditions are equivalent.
   	\begin{itemize}
   		\item[(a)]$\overline{\lk\{U_T^{(m,n)}1_B:(m,n)\in P\re\}}$ is a compact subset of $L^2(X,\mathcal{B}_X,\mu)$.
   		\item[(b)]For any infinite subset $S'=\{(m'_i,n'_i)\}_{i=1}^{\infty}$ of $P$, $$h^{S'}_{\mu}(T,\{B,B^c\})=0.$$
   		\item[(c)]For any infinite subset $S'=\{(m'_i,n'_i)\}_{i=1}^{\infty}$ of $P$ that $\{m'_i\}_{i=1}^{\infty}$ is strictly monotone, $$h^{S'}_{\mu}(T,\{B,B^c\})=0.$$
   	\end{itemize}
   	
   \end{thm}
   \begin{proof}
   	(a) $\Rightarrow$ (b). Let $\eta=\{B,B^c\}$. If $\overline{\lk\{U_T^{(m,n)}1_B:(m,n)\in P\re\}}$ is a compact subset of $L^2(X,\mathcal{B}_X,\mu)$, then for any infinite subset $S'=\{(m'_i,n'_i)\}_{i=1}^{\infty}$ of $P$,  $$\overline{\lk\{U_T^{(m',n')}1_B:(m',n')\in S'\re\}}$$ is a compact subset of $L^2(X,\mathcal{B}_X,\mu)$.  So for any $\delta >0$, there exists $s\in \mathbb{N}$ such that for any $(m'_i,n'_i)\in S'$,$$\mu\lk(T^{-(m'_i,n'_i)}B\Delta T^{-(m'_{j_i},n'_{j_i})}B\re)=\|U_T^{(m'_i,n'_i)}1_B-U_T^{(m'_{j_i},n'_{j_i})}1_B\|_2<\delta$$ for some $j_i\in \{1,2,\ldots,s\}$. It follows from Lemma \ref{lem-2} that for any $\epsilon>0$ and $ (m'_i,n'_i)\in S'$, there exists $ j_i\in \{1,2,\dots,s\}$ such that $$H_\mu\lk(T^{-(m'_i,n'_i)}\eta|T^{-(m'_{j_i},n'_{j_i})}\eta\re)+H_\mu\lk(T^{-(m'_{j_i},n'_{j_i})}\eta|T^{-(m'_i,n'_i)}\eta\re)<\epsilon.$$
   	Thus for any $ i>s$,   $$H_\mu\lk(T^{-(m'_i,n'_i)}\eta|\bigvee_{j=1}^{i-1}T^{-(m'_{j},n'_{j})}\eta\re)\leq H_\mu\lk(T^{-(m'_i,n'_i)}\eta|T^{-(m'_{j_i},n'_{j_i})}\eta\re)<\epsilon.$$
   	We conclude that
   	 $$h^{S'}_{\mu}(T,\eta)=\limsup\limits_{n\to \infty}\frac{1}{n}\sum_{i=2}^{n}H_\mu\lk(T^{-(m'_i,n'_i)}\eta|\bigvee_{j=1}^{i-1}T^{-(m'_{j},n'_{j})}\eta\re)\leq\epsilon.$$ Let $\epsilon \rightarrow 0.$ We obtain that $h^{S'}_{\mu}(T,\eta)=0.$
   	 
   	(b) $\Rightarrow$ (c). This is obvious.

   	(c) $\Rightarrow$ (a). If $\overline{\lk\{U_T^{(m,n)}1_B:(m,n)\in P\re\}}$ is not a compact subset of $L^2(X,\mathcal{B}_X,\mu)$, then there exists $\epsilon>0$ and an infinite subset $F$ of $P$ such that for all $ (m,n),(s,t)\in F$, 
   	\begin{equation}\label{2}
   	\mu\lk(T^{-(m,n)}B\Delta T^{-(s,t)}B\re)=\|U_T^{(m,n)}1_B-U_T^{(s,t)}1_B\|_2\geq \epsilon.
   	\end{equation}  

   	Similar to \eqref{S1}, we obtain an infinite subset $S_F=\{(m_i,n_i)\}_{i=1}^{\infty}$ of $F$ such that  $$\lim\limits_{i\to \infty}\langle g,U_T^{(m_i,n_i)}f\rangle =\langle g,S_F(f)\rangle $$ for any  $f,g\in \mathcal{H}.$ It follows from Lemma \ref{lem-6} that there exists an infinite subset $S'_F=\{(m'_i,n'_i)\}_{i=1}^{\infty}$ of $S_F$ such that $\{m'_i\}_{i=1}^{\infty}$ is strictly monotone and 
   	$$h^{S'_F}_{\mu}(T,\{B,B^c\})\geq \int_X\lk(-S_F(1_B)\log{(S_F(1_B))}-S_F(1_{B^c})\log{(S_F(1_{B^c}))}\re)d\mu.$$
   	As $h_{\mu}^{S'_F}(T,\{B,B^c\})=0$, it follows that $-S_F(1_B)\log{S_F(1_B)}=0$ for $\mu$-a.e. $x\in X$. So $S_F(1_B)$ is an indicator function and $$\langle 1_X,S_F(1_B)\rangle=\|S_F(1_B)\|_{2}.$$
   	Since $$\langle 1_X,1_B\rangle=\lim\limits_{i\to \infty}\langle 1_X,U_T^{(m_i,n_i)}1_B\rangle=\langle 1_X,S_F(1_B)\rangle,$$ it follows that $$\|S_F(1_B)\|_{2}=\|1_B\|_{2}.$$ Hence $$\lim\limits_{i\to \infty}\|U_T^{(m_i,n_i)}1_B-S_F(1_B)\|_{2}= 0.$$ This implies that for sufficiently large $i$ and $j$, $$\mu\lk(T^{-(m_i,n_i)}B\Delta T^{-(m_j,n_j)}B\re)< \frac{\epsilon}{2},$$ which contradicts \eqref{2}.
   	Therefore $\overline{\lk\{U_T^{(m,n)}1_B:(m,n)\in P\re\}}$ is a compact subset of $L^2(X,\mathcal{B}_X,\mu)$.
   \end{proof}
   \medskip

The following description of directional Kronecker algebra is obtained immediately from Theorem \ref{thm-11}.
\begin{cor}\label{prop-111}
	Let  $(X,\mathcal{B}_X,\mu,T)$ be a $\mathbb{Z}^2$-MPS, $\vec{v}=(1,\beta)\in\mathbb{R}^2$ be a direction vector and $b\in (0,\infty)$. Then for any $B\in\mathcal{B}_X$ the following statements are equivalent.
	\begin{itemize}
		\item[(a)]$B\in \mathcal{K}^{\vec{v}}_\mu$.
		\item[(b)]For any infinite subset $S=\{(m_i,n_i)\}_{i=1}^{\infty}$ of $\Lambda^{\vec{v}}(b)$, $$h^{S}_{\mu}(T,\{B,B^c\})=0.$$
		\item[(c)]For any infinite subset $S=\{(m_i,n_i)\}_{i=1}^{\infty}$ of  $\Lambda^{\vec{v}}(b)$ that $\{m_i\}_{i=1}^{\infty}$ is strictly monotone, $$h^{S}_{\mu}(T,\{B,B^c\})=0.$$
	\end{itemize}
Moreover, for any finite measurable partition $\alpha\subset \mathcal{K}^{\vec{v}}_{\mu}$ of $X$ and any infinite subset $S=\{(m_i,n_i)\}_{i=1}^{\infty}$ of $ \Lambda^{\vec{v}}(b),$
$$h^{S}_{\mu}(T,\alpha)=0.$$
\end{cor}

\begin{rem}
By Corollary $\ref{prop-111}$, we immediately obtain 
$$\mathcal{K}_\mu^{\vec{v}}=\left\{B\in\mathcal{B}_X: \text{for any infinite subset }S\text{ of }  \Lambda^{\vec{v}}(b),\text{ } h^{S}_{\mu}(T,\{B,B^c\})=0 \right\}.$$

\end{rem}




\medskip

\section{Direcitonal discrete spectrum systems and null systems}
    \medskip
    In this section, we prove Theorem \ref{thm-3}, which is divided into two lemmas. Moreover,  we prove the directional version of Ku\v shnirenko theorem, that is, Theorem \ref{thm2}. For this purpose, we begin with the following property of directional sequence entropy.

    \begin{thm}\label{thm-1}
	Let  $(X,\mathcal{B}_X,\mu,T)$ be a $\mathbb{Z}^2$-MPS, $\vec{v}=(1,\beta)\in\mathbb{R}^2$ be a direction vector and $b\in (0,\infty)$. Suppose that $\{\alpha_n\}_{n=1}^{\infty}$ satisfying $\alpha_n\nearrow\mathcal{B}_X$, is a family of finite measurable partitions of $X$. Then for any infinite subset $S=\{(m_i,n_i)\}_{i=1}^{\infty} $ of $ \Lambda^{\vec{v}}(b),$ $$ \lim\limits_{n\to \infty}h^{S}_{\mu}(T,\alpha_n)=h^{S}_{\mu}(T).$$
	\end{thm}
\begin{proof}
For any finite measurable partition $\eta$ of $X$ and $k,n \in \mathbb{N}$, 
\begin{align*}
H_\mu \left(\bigvee_{i=1}^k T ^{-(m_i,n_i)} \eta\right)
&\leq H_\mu \left(\bigvee_{i=1}^k T^{-(m_i,n_i)} (\eta \vee \alpha_n)\right)\\
&= H_\mu \left(\bigvee_{i=1}^k T^{-(m_i,n_i)} \alpha_n\right)
 + H_\mu \left(\bigvee_{i=1}^k T^{-(m_i,n_i)} \eta|\bigvee_{i=1}^k T ^{-(m_i,n_i)} \alpha_n\right)\\
&\leq H_\mu \lk(\bigvee_{i=1}^k T^{-(m_i,n_i)} \alpha_n\re) + \sum_{i=1}^k H_\mu \lk(T^{-(m_i,n_i)} \eta|\bigvee_{i=1}^k T^{-(m_i,n_i)} \alpha_n\re)\\
&\leq H_\mu \lk(\bigvee_{i=1}^k T^{-(m_i,n_i)} \alpha_n\re)+ kH_\mu\lk(\eta | \alpha_n\re).
\end{align*}
Divided  by $k$ from both sides and take the upper limits when $k\to \infty$, we can get $$h^{S}_{\mu}(T,\eta)\leq h^{S}_{\mu}(T,\alpha_n)+H_\mu(\eta|\alpha_n).$$
 By Martingale convergence theorem \cite{G},   $$\lim\limits_{n\to \infty} H_\mu(\eta|\alpha_n)= H_\mu(\eta|\mathcal{B}_X)=0.$$ It follows that $$h^{S}_{\mu}(T,\eta)\leq \liminf\limits_{n\to \infty} h^{S}_{\mu}(T,\alpha_n)\leq \limsup\limits_{n\to \infty} h^{S}_{\mu}(T,\alpha_n) \leq h^{S}_{\mu}(T).$$ Therefore
 \begin{equation*}
 	h^{S}_{\mu}(T) =\sup\limits_\eta h^{S}_{\mu}(T,\eta)=\lim\limits_{n\to \infty}h^{S}_{\mu}(T,\alpha_n),
 \end{equation*}
 where the supremum is taken over all finite measurable partitions of $X$
\end{proof}

\medskip	
Now we are able to prove Theorem \ref{thm-3}. For clarity, let us divide the proof into two lemmas.
\begin{lem} \label{lem-4}
Let $(X,\mathcal{B}_X,\mu,T)$ be a $\mathbb{Z}^2$-MPS, $\vec{v}=(1,\beta)\in\mathbb{R}^2$ be a  direction vector and $b\in (0,\infty)$. For any finite measurable partition $\alpha$ of $X$ and  infinite subset $S=\{(m_i,n_i)\}_{i=1}^{\infty}$ of $\Lambda^{\vec{v}}(b)$,  $$h^{S}_{\mu}(T,\alpha)\leq H_\mu(\alpha|\mathcal{K}^{\vec{v}}_\mu).$$
\end{lem}
\begin{proof}
Since $(X,\mathcal{B}_X)$ is separable, there exist countably many finite measurable partitions  $\{\eta_k\}_{k\in\mathbb{N}}\subset \mathcal{K}^{\vec{v}}_\mu $ such that $$\lim\limits_{k \rightarrow \infty }H_{\mu}(\alpha|\eta_k)=H_\mu(\alpha|\mathcal{K}^{\vec{v}}_\mu).$$ For a fixed $k\in\mathbb{N}$ and an infinite subset $S=\{(m_i,n_i)\}_{i=1}^{\infty}$ of $\Lambda^{\vec{v}}(b)$, one has 
\begin{equation}\label{4.2}
\lim\limits_{l\to \infty}\frac{1}{l}H_\mu\left(\bigvee_{i=1}^lT^{-\left(m_i,n_i\right)}\eta_k\right)=0
\end{equation}
by Corollary \ref{prop-111}.
Therefore we conclude
\begin{align*}
h^{S}_{\mu}\left(T,\alpha\right)
&=\limsup\limits_{l\to \infty}\frac{1}{l}H_\mu\left(\bigvee_{i=1}^lT^{-\left(m_i,n_i\right)}\alpha\right)\\
&\leq \limsup\limits_{l\to \infty}{\frac{1}{l}H_\mu\left(\bigvee_{i=1}^lT^{-\left(m_i,n_i\right)}\left(\alpha \vee \eta_k\right)\right)}\\
&\overset{\eqref{4.2}}= \limsup\limits_{l\to \infty}{\frac{1}{l}H_\mu\left(\bigvee_{i=1}^lT^{-\left(m_i,n_i\right)}\left(\alpha \vee \eta_k\right)\right)}-\lim\limits_{l\to \infty}\frac{1}{l}H_\mu\left(\bigvee_{i=1}^lT^{-\left(m_i,n_i\right)}\eta_k\right)\\
&=\limsup\limits_{l\to \infty}{\frac{1}{l}H_\mu\left(\bigvee_{i=1}^lT^{-\left(m_i,n_i\right)}\alpha|\bigvee_{i=1}^lT^{-\left(m_i,n_i\right)}\eta_k\right)}\\
&\leq \limsup\limits_{l\to \infty}{\frac{1}{l}\sum_{i=1}^lH_\mu\left(T^{-\left(m_i,n_i\right)}\alpha|T^{-\left(m_i,n_i\right)}\eta_k\right)}\\
&=H_\mu\left(\alpha|\eta_k\right).
\end{align*}

Finally, let $k\rightarrow\infty$. We get
$$h^{S}_{\mu}(T,\alpha)\leq H_\mu(\alpha|\mathcal{K}^{\vec{v}}_\mu).$$
Now we finish the proof of Lemma \ref{lem-4}.
\end{proof}

For further proof, we need the following result (\cite[Page 69]{RA}).
\begin{lem} \label{lem-5}
Let $(X,\mathcal{B}_X,\mu)$ and $(Y,\mathcal{B}_Y,\nu)$ be two Borel probability spaces and $(X\times Y,\mathcal{B}_{X\times Y},\mu\times \nu)$ be their product Borel probability space. Then for any $B\subset X\times Y$, the following two statements are equivalent.
	\begin{itemize}
	\item[(a)]$B_x:=\{y\in Y:(x,y)\in B\}\in \mathcal{B}_Y$  for $\mu$-a.e. $x\in X$, $B^y:=\{x\in X|(x,y)\in B\}\in \mathcal{B}_X$ for $\nu$-a.e.  $y\in Y$.
	\item[(b)] $B\in \mathcal{B}_{X\times Y}$.
\end{itemize}

\end{lem}
\medskip

\begin{lem}\label{lem-10}
Let $(X,\mathcal{B}_X,\mu,T)$ be a $\mathbb{Z}^2$-MPS, $\vec{v}=(1,\beta)\in\mathbb{R}^2$ be a direction vector and $b\in (0,\infty)$. Then for any finite measurable partition $\alpha$ of $X$, there is an infinite subset $S=\{(m_i,n_i)\}_{i=1}^{\infty}$ of $ \Lambda^{\vec{v}}(b)$ such that $\{m_i\}_{i=1}^{\infty}$ is strictly monotone and $$h^{S}_{\mu}(T,\alpha)\geq H_\mu(\alpha|\mathcal{K}^{\vec{v}}_\mu).$$
\end{lem}
\begin{proof}
To prove our result, we need a $\mathbb{Z}$-MPS, which was introduced by Park \cite{Pa}. Let $$\widetilde{X}=X\times [0,1)^2, \quad\widetilde{\mu}=\mu\times m \text{\quad and\quad} \widetilde{\mathcal{B}}=\mathcal{B}_X\times \mathcal{C},$$ where $\mathcal{C}$ is the Borel  $\sigma$-algebra on $[0,1)^2$ and $m$ is the Lebesgue measure on $[0,1)^2$. Let $$\phi_{s,t}(x,u,v)=(T^{([s+u],[t+v])}x,s+u-[s+u],t+v-[t+v]),$$ where $[a]$ is the integer part of $a$. Write $\phi_{n,n\beta}$ as $W^n$ for any $n\in \mathbb{N}$. Then we get the $\mathbb{Z}$-MPS $(\widetilde{X},\widetilde{\mathcal{B}},\widetilde{\mu},W).$ Let $\mathcal{K_{\widetilde{\mu}}}$ be the Kronecker algebra of $(\widetilde{X},\widetilde{\mathcal{B}},\widetilde{\mu},W).$ We divide the proof into four steps.
 \begin{step} $\mathcal{K_{\widetilde{\mu}}}=\mathcal{K}^{\vec{v}}_{\mu}\times \mathcal{C}$.
\end{step} 	
\begin{proof}[Proof of Step 1.]It is enough to show that $ \mathcal{K_{\widetilde{\mu}}}\subset \mathcal{K}^{\vec{v}}_{\mu}\times \mathcal{C}$ since the opposite side is obviously true. We fix $B\in  \mathcal{K_{\widetilde{\mu}}}$ and put $$B^{(s,t)}=\{x\in X:(x,s,t)\in B\}$$ for $(s,t)\in [0,1)^2$. We are going to show that  $B\in \mathcal{K}^{\vec{v}}_{\mu}\times \mathcal{C}$. By Lemma \ref{lem-5}, it is enough to show that $B^{(s,t)}\in \mathcal{K}^{\vec{v}}_{\mu}$ for $m$-a.e. $(s,t)\in[0,1)^2$. By the definition of $ \mathcal{K}^{\vec{v}}_{\mu}$, it is sufficient to show that $$\overline{\{{U_T^{(m,n)}1_{B^{(s,t)}}:|\beta m- n|<1,(m,n)\in \mathbb{Z}^2\}}}$$ is a compact subset of $L^2(X,\mathcal{B}_X,\mu)$ for $m$-a.e. $(s,t)\in[0,1)^2$.  To do this,
we let
$$f_u(x,s,t)=2^{2u}\int_{\eta^u_{k,l}}1_B(x,\xi,\zeta)d\xi d\zeta,\text{ if $(s,t)\in \eta^u_{k,l}$},$$
where $u\in \mb{N}$ and $\eta^u_{k,l}= [k2^{-u},(k+1)2^{-u})\times [l2^{-u},(l+1)2^{-u})$, $k,l\in \{0,\ldots, 2^u-1\}$.

We have the following claim.
\begin{claim}\label{c-1}The following statements are true.
	\begin{itemize}
		\item[(1)]For each $u\in\mathbb{N}$ and $(s,t)\in[0,1)^2$, $$\overline{\{{U_{T}^{(m,n)}f_u(x,s,t):|\beta m- n|<1,(m,n)\in \mathbb{Z}^2\}}}$$ is a compact subset of $L^2(X,\mathcal{B}_X,\mu)$.
		\item[(2)]  For $m$-a.e. $(s,t)\in[0,1)^2$,
		$$\lim_{u\to\infty}\|f_u(\cdot,s,t)-1_{B^{(s,t)}}\|_{2}=0.$$
	\end{itemize}
\end{claim}
Since the proof of Claim \ref{c-1} is somewhat long and complicated, we move it to Appendix A. Now let us assume that Claim \ref{c-1} holds. Then we deduce that $$\overline{\lk\{{U_{T}^{(m,n)}1_{B^{(s,t)}}:|\beta m- n|<1,(m,n)\in \mathbb{Z}^2}\re\}}$$ is a compact subset of $L^2(X,\mathcal{B}_X,\mu)$ for $m$-a.e. $(s,t)\in[0,1)^2$, since $1_{B^{(s,t)}}\in L^2(X,\mathcal{B}_X,\mu)$ for $m$-a.e. $(s,t)\in[0,1)^2$. By the arguments at begining, we finish the proof of Step 1.
\end{proof}
\begin{step}
For a given measurable partition $\mathfrak{C}=\{D_1,D_2,\ldots, D_r\}$, there exsits an infinite subset $S'=\{(m'_i,n'_i)\}_{i=1}^{\infty}$ of $ \Lambda^{\vec{v}}(b)$ depending on $\mathfrak{C}$ such that $\{m'_i\}_{i=1}^{\infty}$ is strictly monotone and
\begin{align}\label{f-2}
\lim\limits_{i\to\infty}\langle U_T^{(m'_i,n'_i)}(1_{D_j}-\mathbb{E}(1_{D_j}|\mathcal{K}^{\vec{v}}_\mu)),1_B\rangle=0
\end{align}
for any $B\in\mathcal{B}_X$ and $j\in\{1,2,\ldots,r\}.$
\end{step}
\begin{proof}[Proof of Step 2.]

By Theorem \ref{knt},  there exists a sequence $\widetilde{S} \subset \mathbb{Z}_+$ with $d(\widetilde{S})=1$ such that
\begin{align}\label{eq-1}
&\lim\limits_{n\to \infty,n\in \widetilde{S}}\langle U_{W}^n (1_{D_j}\times 1_{[0,1)^2}-\mathbb{E}(1_{D_j}\times 1_{[0,1)^2}|\mathcal{K}_{\widetilde{\mu}})),1_B\times1_{[0,1)^2}\rangle=0
\end{align}
for any $B\in \mathcal{B}_X$ and $j\in \{1,2,\ldots,r\}$, where $1_{D_j}\times 1_{[0,1)^2}=1_{D_j}(x)\cdot1_{[0,1)^2}(s,t)$ is defined on $\widetilde{X}$. We omit the independent variable for convenience.
 By Step 1 and Proposition \ref{prop1}, we have
\begin{align*}
&\mathbb{E}(1_{D_j}\times 1_{[0,1)^2}|\mathcal{K}_{\widetilde{\mu}})=\mathbb{E}(1_{D_j}\times 1_{[0,1)^2}|\mathcal{K}^{\vec{v}}_{\mu}\times \mathcal{C})\\
=&\mb{E}(1_{D_j}|\mathcal{K}^{\vec{v}}_\mu)\times \mb{E}(1_{[0,1)^2}|\mc{C})=\mb{E}(1_{D_j}|\mathcal{K}^{\vec{v}}_\mu)\times 1_{[0,1)^2}.
\end{align*}
Let $f_j=1_{D_j}-\mathbb{E}(1_{D_j}|\mathcal{K}^{\vec{v}}_\mu)$, $j\in \{1,2,\ldots,r\}.$ Then we conclude
$$f_j\times 1_{[0,1)^2}=1_{D_j}\times 1_{[0,1)^2}-\mathbb{E}(1_{D_j}\times 1_{[0,1)^2}|\mathcal{K}_{\widetilde{\mu}})\text{, }j\in \{1,2,\ldots,r\}. $$
By Fubini's theorem,
\begin{equation}\label{f}
\begin{split}
&|\langle U_{W}^n (f_j\times 1_{[0,1)^2}),1_B\times1_{[0,1)^2}\rangle-\langle U_T^{(n,[n\beta])}(f_j\times 1_{[0,1)^2}),1_B\times1_{[0,1)^2}\rangle|\\
=&|\int_{[0,1)^2}\int_{B}\lk(f_j(T^{([n+s],[n\beta+t])}x)-f_j(T^{(n,[n\beta])}x)\re)d\mu(x) dt ds|\\
=&|\int_{[0,1)}\int_{B}\lk(f_j(T^{(n,[n\beta+t])}x)-f_j(T^{(n,[n\beta])}x)\re)d\mu(x) dt |\\
=&|\int_{B}\int_{1-\{n\beta\}}^{1}\lk(f_j(T^{(n,[n\beta+t])}x)-f_j(T^{(n,[n\beta])}x)\re) dtd\mu(x) |\\
\leq&\{n\beta\}\int_{B}|f_j(T^{(n,[n\beta]+1)}x)|+|f_j(T^{(n,[n\beta])}x)|d\mu(x)\\
\leq&2\|f_j\|_{1}\{n\beta\},
\end{split}
\end{equation}
where $\{n\beta\}$ is the decimal part of $n\beta$.
Let $M=2\max_{1\leq j\leq r}\{\|f_j\|_{1}\}$ and $\Theta_\gamma=\{n\in \mathbb{Z}_+:\{n\beta\}\leq \frac{\epsilon}{2^{\gamma}M}\}$ for all  $\gamma\in\mathbb{N}$. Since $\Theta_\gamma$ is a syndetic set (a subset $S$ of $\mathbb{Z}_+$ is syndetic if it has bounded gaps, i.e. there is a number $N\in \mathbb{N}$ such that $\{i,i+1,\dots,i+N\}\cap S\neq \emptyset$ for every $i\in \mathbb{Z}_+$), it follows that $d(\Theta_\gamma\cap \widetilde{S})>0$ for all $\gamma\in\mathbb{N}$. Take $n_\gamma\in \Theta_\gamma\cap \widetilde{S}$ for each $\gamma\in\mathbb{N}$. By \eqref{f} we obtain that
\begin{equation}\label{f1}
\lim\limits_{\gamma\to\infty}|\langle U_{W}^{n_\gamma} (f_j\times 1_{[0,1)^2}),1_B\times1_{[0,1)^2}\rangle-\langle U_T^{(n_\gamma,[n_\gamma\beta])}\lk(f_j\times 1_{[0,1)^2}\re),1_B\times1_{[0,1)^2}\rangle|=0
\end{equation}
for all $j\in \{1,2,\ldots, r\}.$
By the construction of $\Theta_\gamma$, there exists $\gamma_0\in \mathbb{N}$ such that for any $\gamma >\gamma_0$ we have  $(n_\gamma,[n_\gamma\beta])\in \Lambda^{\vec{v}}(b)$.
It follows that there exists an infinite subsequence $\{n_\gamma\}_{\gamma=\gamma_0+1}^{\infty} $ of $ \widetilde{S}$ such that $(n_\gamma,[n_\gamma\beta])\in \Lambda^{\vec{v}}(b)$ and
\begin{align*}
&\lim\limits_{\gamma\to \infty}\langle U_T^{(n_\gamma,[n_\gamma\beta])} (1_{D_j}-\mathbb{E}(1_{D_j}|\mathcal{K}^{\vec{v}}_\mu)),1_B\rangle\\
=&\lim\limits_{\gamma\to \infty}|\langle U_T^{(n_\gamma,[n_\gamma\beta])}\lk(f_j\times 1_{[0,1)^2}\re),1_B\times 1_{[0,1)^2}\rangle|\overset{\eqref{eq-1}\eqref{f1}}=0
\end{align*}
for all $j\in\{1,2,\ldots,r\}.$ Let $$S'=\{(m'_i,n'_i):(m'_i,n'_i)=(n_{\gamma_0+i},[n_{\gamma_0+i}\beta]),i\in\mathbb{N}\},$$ which is the infinite subset we want. 
\end{proof}

 \begin{step}
Let $\eta$ be a  finite measurable partition of $X$. Then there exists an infinite subset $S''=\{(m''_i,n''_i)\}_{i=1}^{\infty} $ of $ \Lambda^{\vec{v}}(b)$
such that $\{m''_i\}_{i=1}^{\infty}$ is strictly monotone and
 $$\liminf_{i\to\infty}H_\mu(T^{-(m''_i,n''_i)}\alpha|\eta)\geq H_\mu(\alpha|\mathcal{K}^{\vec{v}}_\mu).$$
 \end{step}
 \begin{proof}[Proof of Step 3.]
Let $\alpha=\{A_1,A_2,\ldots,A_k\}$ and $\eta=\{B_1,B_2,\ldots,B_l\}$. By Step 2, there exists an infinite subset
$S''=\{(m''_i,n''_i)\}_{i=1}^{\infty} $ of $ \Lambda^{\vec{v}}(b)$
such that $\{m''_i\}_{i=1}^{\infty}$ is strictly monotone and
\begin{align}\label{3}\lim\limits_{i\to\infty}\langle U_T^{(m''_i,n''_i)}(1_{A_p}-\mathbb{E}(1_{A_p}|\mathcal{K}^{\vec{v}}_\mu)),1_{B_q}\rangle=0
\end{align}
for all $p\in \{1,2,\ldots,k\}$ and $q\in\{1,2,\ldots,l\}$. Hence
\begin{align*}
&\liminf\limits_{i\to \infty}{H_\mu(T^{-(m''_i,n''_i)}\alpha|\eta)}\\
=&\liminf\limits_{i\to \infty}{\sum_{p,q}-\mu\lk(T^{-(m''_i,n''_i)}A_p\cap B_q\re)}\log{\left(\frac{\mu(T^{-(m''_i,n''_i)}A_p\cap B_q)}{\mu(B_q)}\right)}\\
=&\liminf\limits_{i\to \infty}{\sum_{p,q}-\langle U_T^{(m''_i,n''_i)}1_{A_p},1_{B_q}\rangle\log{\left(\frac{\langle U_T^{\left(m''_i,n''_i\right)}1_{A_p},1_{B_q}\rangle}{\mu\left(B_q\right)}\right)}}\\
\overset{\eqref{3}}=&\liminf\limits_{i\to \infty}{\sum_{p,q}-\langle U_T^{\left(m''_i,n''_i\right)}\mathbb{E}\left(1_{A_p}|\mathcal{K}^{\vec{v}}_{\mu}\right),1_{B_q}\rangle\log{\left(\frac{\langle U_T^{\left(m''_i,n''_i\right)}\mathbb{E}\left(1_{A_p}|\mathcal{K}^{\vec{v}}_{\mu}\right),1_{B_q}\rangle}{\mu\left(B_q\right)}\right)}}.
\end{align*}
Let $a_{pq}^i=-\langle U_T^{\left(m''_i,n''_i\right)}\mathbb{E}\left(1_{A_p}|\mathcal{K}^{\vec{v}}_{\mu}\right),1_{B_q}\rangle \log{\left(\frac{\langle U_T^{\left(m''_i,n''_i\right)}\mathbb{E}\left(1_{A_p}|\mathcal{K}^{\vec{v}}_{\mu}\right),1_{B_q}\rangle}{\mu\left(B_q\right)}\right)}$.
From concavity of $-x\log{x}$, we deduce that
\begin{align*}
\frac{a_{pq}^i}{\mu\left(B_q\right)}
&=-\left(\int_{B_q}{\frac{U_{T}^{\left(m''_i,n''_i\right)}\mathbb{E}\left(1_{A_p}|\mathcal{K}^{\vec{v}}_{\mu}\right)}{\mu\left(B_q\right)}}\,d\mu\right)\log{\left(\int_{B_q}{\frac{U_{T}^{\left(m''_i,n''_i\right)}\mathbb{E}\left(1_{A_p}|\mathcal{K}^{\vec{v}}_{\mu}\right)}{\mu\left(B_q\right)}}\,d\mu\right)}\\
&\geq -\int_{B_q}{\frac{U_{T}^{\left(m''_i,n''_i\right)}\mathbb{E}\left(1_{A_p}|\mathcal{K}^{\vec{v}}_{\mu}\right)}{\mu\left(B_q\right)}}\log{\left(U_{T}^{\left(m''_i,n''_i\right)}\mathbb{E}\left(1_{A_p}|\mathcal{K}^{\vec{v}}_{\mu}\right)\right)}\,d\mu.
\end{align*}
That is, $$a_{pq}^i\geq -\int_{B_q}{U_T^{\left(m''_i,n''_i\right)}\mathbb{E}\left(1_{A_p}|\mathcal{K}^{\vec{v}}_{\mu}\right)}\log{\left(U_T^{\left(m''_i,n''_i\right)}\mathbb{E}\left(1_{A_p}|\mathcal{K}^{\vec{v}}_{\mu}\right)\right)}\,d\mu.$$
We conclude that
\begin{align*}
\sum_{p,q}a_{pq}^i
&\geq \sum_{p,q}-\int_{B_q}{U_T^{\left(m''_i,n''_i\right)}\mathbb{E}\left(1_{A_p}|\mathcal{K}^{\vec{v}}_{\mu}\right)}\log{\left(U_T^{\left(m''_i,n''_i\right)}\mathbb{E}\left(1_{A_p}|\mathcal{K}^{\vec{v}}_{\mu}\right)\right)}\,d\mu\\
&=\sum_p -\int_X{U_T^{\left(m''_i,n''_i\right)}\mathbb{E}\left(1_{A_p}|\mathcal{K}^{\vec{v}}_{\mu}\right)}\log{\left(U_T^{\left(m''_i,n''_i\right)}\mathbb{E}\left(1_{A_p}|\mathcal{K}^{\vec{v}}_{\mu}\right)\right)}\,d\mu\\
&=\sum_p -\int_X{\mathbb{E}\left(1_{A_p}|\mathcal{K}^{\vec{v}}_{\mu}\right)}\log{\left(\mathbb{E}\left(1_{A_p}|\mathcal{K}^{\vec{v}}_{\mu}\right)\right)}\,d\mu.
\end{align*}
Therefore
$$\liminf\limits_{i\to \infty}{H_\mu\left(T^{-\left(m''_i,n''_i\right)}\alpha|\eta\right)} \geq H_\mu\left(\alpha|\mathcal{K}^{\vec{v}}_\mu\right).$$
 \end{proof}
 \begin{step}
	In this step, we finish the proof of Lemma \ref{lem-10}.
\end{step}

We can define inductively an infinite subset $S=\{(m_i,n_i)\}_{i=1}^{\infty}$  of $\Lambda^{\vec{v}}(b)$ by using Step 3 repeatedly such that $\{m_i\}_{i=1}^{\infty}$ is strictly monotone and for any $i\in \mathbb{N}$,
$$H_\mu(T^{-(m_i,n_i)}\alpha|\bigvee_{j=1}^{i-1} T^{-(m_j,n_j)}\alpha)\geq H_\mu(\alpha|\mathcal{K}^{\vec{v}}_\mu)-\frac{1}{2^i}.$$
Therefore, one has
\begin{align*}
h^{S}_{\mu}(T, \alpha)
&=\limsup_{k\to \infty}{\frac{1}{k}\sum_{i=1}^{k}H_\mu(T^{-(m_i,n_i)}\alpha|\bigvee_{j=1}^{i-1} T^{-(m_j,n_j)}\alpha)}\\
&\geq\limsup_{k\to \infty}{\frac{1}{k}\sum_{i=1}^{k}(H_\mu(\alpha|\mathcal{K}^{\vec{v}}_\mu)-\frac{1}{2^i})}\\
&=H_\mu(\alpha|\mathcal{K}^{\vec{v}}_\mu).
\end{align*}
This finishes the proof of Theorem \ref{lem-10}.
\end{proof}

\begin{rem}
By Lemma \ref{lem-4} and Lemma \ref{lem-10}, we complete the proof of Theorem \ref{thm-3}.
\end{rem}

By Lemma \ref{lem-4} and Lemma \ref{lem-10}, one of our main results, the directional version of Ku\v shnirenko theorem, is proved as follows.
\begin{proof}[Proof of Theorem \ref{thm2}] Let $(X,\mathcal{B}_X,\mu,T)$ be a $\mathbb{Z}^2$-MPS and $\vec{v}=(1,\beta)\in\mathbb{R}^2$ be a direction vector.
	
(a) $\Rightarrow$ (b). Assmue that $\mu$ has $\vec{v}$-discrete spectrum system. Then $\mathcal{K}^{\vec{v}}_\mu=\mc{B}_X$. Given $b\in(0,\infty)$. By Lemma \ref{lem-4},
  $$h^{S}_{\mu}(T,\alpha)\leq H_\mu(\alpha|\mathcal{K}^{\vec{v}}_\mu)=0$$
 for any infinite subset $S=\{(m_i,n_i)\}_{i=1}^{\infty}$ of $\Lambda^{\vec{v}}(b)$, which implies $(X,\mathcal{B}_X,\mu,T)$ is $\vec{v}$-null.
 
(b) $\Rightarrow$ (a). Assume that $(X,\mathcal{B}_X,\mu,T)$ is $\vec{v}$-null. Then for any $b\in (0,\infty)$ and infinite subset $S=\{(m_i,n_i)\}_{i=1}^{\infty}$ of $\Lambda^{\vec{v}}(b)$, $$h^{S}_{\mu}(T)=0.$$ Hence
 $$h^{S}_{\mu}(T,\{B,B^c\})=0$$
 for any $B\in \mathcal{B}_X$, which implies $B\in \mathcal{K}^{\vec{v}}_\mu$ by Corollary \ref{prop-111}. Therefore $\mathcal{K}^{\vec{v}}_\mu=\mc{B}_X$, that is, $\mu$ has $\vec{v}$-discrete spectrum.
 This completes the proof of Theorem \ref{thm2}.

\end{proof}

\section{Sequence entropy and direcitonal sequence entropy}
In this seciton, we recall and prove Theorem \ref{thm2} and Theorem \ref{thm3}, which establish the relation between sequence entropy (respectively discrete spectrum) and directional sequence entropy (respectively directional discrete spectrum). For these purposes, we recall a result in \cite[Lemma 5.1]{KL}, which is restated as follows.
\begin{lem}\label{11}
Let $(X,\mathcal{B}_X,\mu,T)$ be a $\mathbb{Z}^2$-MPS. Then the following conditions are equivalent.
\begin{itemize}
\item[(a)]$\overline{\lk\{U_T^{(m,n)}1_B:(m,n)\in \mathbb{Z}^2\re\}}$ is a compact subset of $L^2(X,\mathcal{B}_X,\mu)$.
\item[(b)]For any infinite subset $S=\{(m_i,n_i)\}_{i=1}^{\infty}$ of $\mathbb{Z}^2$, $$h^{S}_{\mu}(T,\{B,B^c\})=0.$$
\end{itemize}
\end{lem}

Next we prove a combinatorial result as follows.

\begin{lem}\label{c}
Let  $\vec{v}=(1,\beta_1), \vec{w}=(1,\beta_2)\in \mb{R}^2$ be two directions with $\beta_1\neq \beta_2$. Then $$\mb{Z}^2=\Lambda^{\vec{v}}(b)+\Lambda^{\vec{w}}(b)$$ for any $b> 4([|\beta_1-\beta_2|]+1)$, where $$\Lambda^{\vec{v}}(b)+\Lambda^{\vec{w}}(b)=\{(m_1+m_2,n_1+n_2):(m_1,n_1)\in \Lambda^{\vec{v}}(b)\text{ and } (m_2,n_2)\in \Lambda^{\vec{w}}(b)\}.$$
\begin{proof}
Given $(m,n)\in\mb{Z}^2$. By division algorithm,  we can choose $m_1\in\mb{Z}$ such that 
\begin{equation}\label{c1}
|n-\beta_2m+(\beta_1-\beta_2)m_1|\leq |\beta_1-\beta_2|.
\end{equation}
Since $b/4>[|\beta_1-\beta_2|]+1$, we can choose $n_1\in \mb{Z}$ such that 
\begin{itemize}
	\item[(a)]$\beta_1m_1-b/2\leq n_1\leq \beta_1m_1+b/2$,
	\item[(b)]$n_1=\beta_1m_1+r$  with  $|r|\leq b/4$.
\end{itemize} 

Let $n_2=n-n_1$ and $m_2=m-m_1$. Then it is sufficient to prove that $$(m_2,n_2)\in \Lambda^{\vec{w}}(b).$$
Note that 
\begin{equation*}
\begin{split}
|n_2-\beta_2m_2|=&|n-\beta_2m+n_1-\beta_2m_1|=|n-\beta_2m+\beta_1m_1+r-\beta_2m_1|\\
=&|n-\beta_2m+(\beta_1-\beta_1)m_1+r|\overset{\eqref{c1}}\leq|\beta_1-\beta_2|+|r|\\
\leq& ([|\beta_1-\beta_2|]+1)+b/4\leq b/2,
\end{split}
\end{equation*}
which implies $(m_2,n_2)\in \Lambda^{\vec{w}}(b)$.
This finishes the proof of Lemma \ref{c}.
\end{proof}
\end{lem}

With the help of the above lemma, we are able to prove Theorem \ref{thm2} and Theorem \ref{thm3}, which are recalled and proved as follows.
\begin{thm}\label{111}
	Let $(X,\mathcal{B}_X,\mu,T)$ be a $\mathbb{Z}^2$-MPS. Given a finite measurable partition $\alpha$ of $X$. If there exists an infinite subset $S$ of $\mathbb{Z}^2$ such that $h^S_{\mu}(T,\alpha)>0$, then there is at most one direction $\vec{v}=(1,\beta)\in \mb{R}^2$ such that $$h_\mu^{\vec{v},*}(\alpha)=0,$$ where $h_\mu^{\vec{v},*}(\alpha)$ is defined by \eqref{4}.
	\begin{proof}
		Given a finite measurable partition $\alpha$ of $X$. Assume that there exists an infinite subset $S$ of $\mathbb{Z}^2$ such that $$h^S_{\mu}(T,\alpha)>0.$$ 
		
		In the following we show that there is at most one direction $\vec{\tau}=(1,\beta)\in \mb{R}^2$ such that $h_\mu^{\vec{\tau},*}(\alpha)=0.$ If this is not true then there exist two directions $\vec{v}=(1,\beta_1), \vec{w}=(1,\beta_2)\in \mb{R}^2$ with $\beta_1\neq \beta_2$ such that $$h_\mu^{\vec{v},*}(\alpha)=0\quad\text{and}\quad h_\mu^{\vec{w},*}(\alpha)=0.$$ By  Lemma \ref{lem-10} and \eqref{4}, we obtian that $$H_\mu(\alpha|\mathcal{K}_\mu^{\vec{v}})=H_\mu(\alpha|\mathcal{K}_\mu^{\vec{w}})=0,$$ which implies that  $\alpha\subset \mathcal{K}_\mu^{\vec{v}} \bigcap  \mathcal{K}_\mu^{\vec{w}}$. Given $B\in\alpha$. In the following we show that $$\overline{\{U_T^{(m,n)}1_B:(m,n)\in \mathbb{Z}^2\}}$$
		is a compact subset of $L^2(X,\mc{B}_X,\mu)$, which implies that $h^{S'}_{\mu}(T,\{B,B^c\})=0$ for any infinite subset $S'=\{(m'_i,n'_i)\}_{i=1}^{\infty}$ of $\mathbb{Z}^2$ by Lemma \ref{11}. In fact, by Lemma \ref{c}, taking $b=([|\beta_1-\beta_2|]+2)$, we have $$\Lambda^{\vec{v}}(b)+\Lambda^{\vec{w}}(b)=\mb{Z}^2$$ and hence it suffices to prove that $$\mathcal{R}_b:=\overline{\{U_T^{(m,n)}1_B:(m,n)\in \Lambda^{\vec{v}}(b)+\Lambda^{\vec{w}}(b)\}}$$ is a compact subset of $L^2(X,\mc{B}_X,\mu)$. Note that $B\in \mathcal{K}_\mu^{\vec{v}} \bigcap  \mathcal{K}_\mu^{\vec{w}}$. So $$\mathcal{P}_b:=\overline{\{U_T^{(m,n)}1_B:(m,n)\in \Lambda^{\vec{v}}(b)\}} \quad\text{and}\quad \mathcal{Q}_b:=\overline{\{U_T^{(m,n)}1_B:(m,n)\in \Lambda^{\vec{w}}(b)\}}$$  are  compact subsets of $L^2(X,\mathcal{B}_X,\mu)$. 	For any $\epsilon>0$, let $$\{(m_i,n_i)\}_{i=1}^s\subset \Lambda^{\vec{v}}(b)\quad\text{and}\quad \{(u_j,v_j)\}_{j=1}^s\subset \Lambda^{\vec{w}}(b)$$ be $\epsilon/2$-nets of $\mathcal{P}_b$ and $\mathcal{Q}_b$ in $L^2(X,\mathcal{B}_X,\mu)$, respectively. Hence for any $(p_1,q_1)\in \Lambda^{\vec{v}}(b)$ and $(p_2,q_2)\in \Lambda^{\vec{w}}(b)$, we conclude
		$$\|U_T^{(p_1,q_1)}1_{B}-U_T^{(m_i,n_i)}1_{B}\|_2<\epsilon/2 \quad\text{and}\quad \|U_T^{(p_2,q_2)}1_{B}-U_T^{(u_j,v_j)}1_{B}\|_2<\epsilon/2$$
		for some $i,j\in\{1,\ldots,s\}$. Moreover, we deduce
		\begin{equation}\label{8}
		\begin{split}
		&\|U_T^{(p_1+p_2,q_1+q_2)}1_{B}-U_T^{(m_i+u_j,n_i+v_j)}1_{B}\|_2\\
	    \leq& \|U_T^{(p_1,q_1)}1_{T^{-(p_2,q_2)}B}-U_T^{(m_i,n_i)}1_{T^{-(p_2,q_2)}B}\|_2\\
	    &+\|U_T^{(p_2,q_2)}1_{T^{-(m_i,n_i)}B}-U_T^{(u_j,v_j)}1_{T^{-(m_i,n_i)}B}\|_2\leq \epsilon.
		\end{split}
		\end{equation}
		It follows from \eqref{8} that $$\Theta_b:=\{(m_i+u_j,n_i+v_j):1\leq i,j\leq s\}$$ is a finite $\epsilon$-net of $\mathcal{R}_b$ in $L^2(X,\mathcal{B}_X,\mu)$, which implies that $\mathcal{R}_b$ is a compact subset of $L^2(X,\mathcal{B}_X,\mu)$. By the arguments at begining of this proof, we obtain $$h^{S'}_{\mu}(T,\{B,B^c\})=0$$ for any infinite subset $S'=\{(m'_i,n'_i)\}_{i=1}^{\infty}$ of $\mathbb{Z}^2$. 
		
		Since $\bigvee_{B\in \alpha}\{B,B^c\}$ is finer than $\alpha$, it follows  that $$h^{S'}_{\mu}(T,\alpha)=0$$ for any infinite subset $S'$ of $\mathbb{Z}^2$, which contradicts the fact that there exists an infinite subset $S$ of $\mathbb{Z}^2$ such that $h^S_{\mu}(T,\alpha)>0$. This completes the proof of Theorem \ref{111}.
	\end{proof}
\end{thm}
The following example shows that  they could both hold at the same time that there is an infinite subset $S$ of $\mathbb{Z}^2$ such that $h^S_{\mu}(T,\alpha)>0$ and there is a direction vector $\vec{v}=(1,\beta)\in \mathbb{R}^2$ such that $h_\mu^{\vec{v},*}(\alpha)=0$.

\begin{exam}
Let $(Y,2^Y,\mu)$ denote the measure space, where $Y=\{0,1\}$, $2^Y$ is the collection consisting of all subsets of $Y$ and the points $0,1$ have measure $1/2$. Let $$(X,\mathcal{B}_X,m)=\prod_{-\infty}^{\infty}(Y,2^Y,\mu).$$ Define $T_1=Id_X:X\to X$ by $$T_1(\{x_n\})=\{x_n\}$$ for any $\{x_n\}\in X$.
Define $T_2:X\to X$ by $$T_2(\{x_n\})=\{y_n\}$$ where $y_n=x_{n+1}$ for all $n\in\mb{N}$, that is, $T_2$ is the two-sided $(\frac{1}{2},\frac{1}{2})$-shift.  Define $\mathbb{Z}^2$-actions $T$ by $$T^{(m,n)}=T_1^mT_2^n.$$
Then we obtain a $\mathbb{Z}^2$-MPS $(X,\mc{B}_X,m,T).$

Let $$\alpha=\{_0[j]_0:j=0,1\}$$ be a finite measurable partition of $X$, where $_0[j]_0=\{x=\{x_n\}:x_0=j\}.$ On the one hand, take $S=\{(0,n)\}_{n=1}^{\infty}$. It is known that (see \cite[Page 102, Theorem 4.26]{Peter}) 
\begin{equation}
\begin{split}
h_{\mu}^S(T,\alpha)=\limsup_{n\to \infty}\frac{1}{n}H_{\mu}\lk(\bigvee_{i=1}^{n}T^{-(0,i)}\alpha\re)=\log{2}>0.
\end{split}
\end{equation}

On the other hand, let $\vec{v}=(1,0)$. Since $T_1=Id_X$, it is easy to see that $$h_\mu^{\vec{v},*}(\alpha)=0.$$
\end{exam}
It is known that for a $\mathbb{Z}^q$-MPS $(X,\mathcal{B}_X,\mu,T)$, $\mu$ has discrete spectrum if and only if $\mathcal{B}_X=\mathcal{K}_\mu$, where $\mathcal{K}_\mu$ is the Kronecker algebra of $(X,\mathcal{B}_X,\mu,T)$ (see \cite{G}).
Following the proof of Theorem \ref{111}, we can immediately obtain the relation between classical discrete spectrum and directional discrete spectrum.  
\begin{thm}	\label{thm11}
			Let $(X,\mathcal{B}_X,\mu,T)$ be a $\mathbb{Z}^2$-MPS. Then the following statements are equivalent.
			\begin{itemize}
				\item[(a)] There exist two directions $\vec{v}=(1,\beta_1),\vec{w}=(1,\beta_2)\in \mb{R}^2$ with $\beta_1\neq \beta_2$ such that $\mu$ has $\vec{v}$-discrete spectrum and $\vec{w}$-discrete spectrum.
				\item[(b)]$\mu$ has discrete spectrum.
			\end{itemize}
	\begin{proof}
	(a) $\Rightarrow$ (b). If there exist two direction $\vec{v}=(1,\beta_1)$ and $\vec{w}=(1,\beta_2)$ with $\beta_1\notin \beta_2$ such that $\mu$ has $\vec{v}$-discrete spectrum and $\vec{w}$-discrete spectrum, that is,
	$$\mathcal{K}_\mu^{\vec{v}}=\mc{B}_X=\mathcal{K}_\mu^{\vec{w}}.$$ Given $b\in(0,\infty)$ large enough. For any $B\in\mc{B}_X$, $$\mathcal{P}_b:=\overline{\{U_T^{(m,n)}1_B:(m,n)\in \Lambda^{\vec{v}}(b)\}} \text{ and } \mathcal{Q}_b:=\overline{\{U_T^{(m,n)}1_B:(m,n)\in \Lambda^{\vec{w}}(b)\}}$$  are  compact subsets of $L^2(X,\mathcal{B}_X,\mu)$. By the proof of Theorem \ref{111}, we deduce $$\overline{\{U_T^{(m,n)}1_B:(m,n)\in \mathbb{Z}^2\}}$$ is a compact subset of $L^2(X,\mathcal{B}_X,\mu)$, which implies $\mathcal{K}_\mu=\mc{B}_X$.
	That is $\mu$ has discrete spectrum.
	
	(b) $\Rightarrow$ (a). If $\mu$ has discrete spectrum, then $\mathcal{K}_\mu=\mc{B}_X$, that is, $$\overline{\{U_T^{(m,n)}1_B:(m,n)\in \mathbb{Z}^2\}}$$ is a compact subset of $L^2(X,\mathcal{B}_X,\mu)$ for any $B\in\mc{B}_X$. Since for any $b\in(0,\infty)$, $\mathcal{P}_b$ and $\mathcal{Q}_b$ are closed subsets of $\overline{\{U_T^{(m,n)}1_B:(m,n)\in \mathbb{Z}^2\}}$, it follows that $\mathcal{P}_b$ and $\mathcal{Q}_b$ are compact subsets of $L^2(X,\mathcal{B}_X,\mu)$, which implies that 
	$$\mathcal{K}_\mu^{\vec{v}}=\mc{B}_X=\mathcal{K}_\mu^{\vec{w}}.$$
	Therefore $\mu$ has $\vec{v}$-discrete spectrum and $\vec{w}$-discrete spectrum.
	\end{proof}
\begin{rem}
	By the proof of Theorem \ref{thm11}, it is clear that if $\mu$ has discrete spectrum then $\mu$ has $\vec{v}$-discrete spectrum for any direction $\vec{v}=(1,\beta)\in \mb{R}^2$.
\end{rem}
\end{thm}
\section*{Acknowledgement}
The authors would like to thank Wen Huang and Xiangdong Ye for making many valuable suggestions. C. Liu was partially supported by NNSF of China (12090012, 12031019, 11731003), L. Xu was partially supported by NNSF of China (11801538, 11871188, 12031019) and the USTC Research Funds of the Double First-Class Initiative.

\begin{appendix}
	\section{Proof of Claim 1}
In this section, we complete the proof of Claim 1.	
	
It follows from Lebesgue differentiation theorem (see \cite[Page 106, Corollary 1.5]{st}) that we obtain the following result.
\begin{lem} \label{lem-41}
	For any $E\in \mathcal{C}$ with $m(E)>0$, we deduce
	\begin{align}\label{A}
	\lim\limits_{u\to \infty}\frac{m(\eta^u(s,t)\bigcap E)}{m(\eta^u(s,t))}=1
	\end{align}
	for $m$-a.e. $(s,t)\in E$, 	where $\eta^u(s,t)$ is the cube $\eta^u_{k,l}$ containing $(s,t)$.
\end{lem}
By Lemma \ref{lem-41}, we can prove the following consequence.
\begin{lem} \label{lem-42}
	For $\mu$-a.e. $x\in X$, $$\lim\limits_{u\to\infty}f_u(x,s,t)=1_B(x,s,t)$$ for $m$-a.e. $(s,t)\in [0,1)^2$.
\end{lem}
\begin{proof} It is clear that $f_u(x,s,t)$ is measurable on $[0,1)^2$ for $\mu$-a.e. $x\in X$. For a given $g(s,t)\in L^{\infty}\lk([0,1)^2,\mathcal{C},m\re)$, we let
	$$g_u(s,t)=2^{2u}\int_{\eta^u_{k,l}}g(\xi,\zeta)d\xi d\zeta,\text{ if $(s,t)\in \eta^u_{k,l}$},$$
	where $u\in \mb{N}$ and $\eta^u_{k,l}= [k2^{-u},(k+1)2^{-u})\times [l2^{-u},(l+1)2^{-u})$ for all $k,l\in \{0,\ldots, 2^u-1\}$. To prove this result, we prove a stronger result, that is, for any $g(s,t)\in L^{\infty}([0,1)^2,\mathcal{C},m)$, we have  $\lim\limits_{u\to \infty} g_u(s,t)=g(s,t)$ for $m$-a.e. $(s,t)\in [0,1)^2$.
	
	First, we assume that $g(s,t)\in C\lk([0,1)^2\re)$, where $C\lk([0,1)^2\re)$ represents the set of all continuous functions defined on $[0,1)^2$. For any  $\epsilon>0$, there exists $\delta>0$ such that for any $(s,t),(s',t')\in [0,1)^2$ with $|(s,t)-(s',t')|<\delta$,  $$|g(s,t)-g(s',t')|<\epsilon.$$ So for any $u$ large enough, we obtain that
	\begin{align*}
	|g_u(s,t)-g(s,t)|&=|2^{2u}\int_{\eta^u(s,t)}g(\xi,\zeta)-g(s,t)d\xi d\zeta|\\
	&\leq 2^{2u}\int_{\eta^u(s,t)}|g(\xi,\zeta)-g(s,t)|d\xi d\zeta \leq \epsilon,
	\end{align*}
	where $I_u(s,t)$ is the one that contains $(s,t)$ in  $\{\eta^u_{k,l}:k,l\in \{0,\ldots, 2^u-1\}\}$. It follows that $\lim\limits_{u\to \infty}g_u(s,t)= g(s,t)$ pointwise.

	Now we prove the general case. Fix a $g(s,t)\in L^{\infty}([0,1)^2,\mathcal{C},m)$. For any $q\in \mathbb{N}$, taking $\epsilon_q=\frac{1}{q}$, by Lusin theorem, there exists a measurable closed set $E_q\subset [0,1)^2$ such that $E_q\subset E_{q+1}$, $m(E_q)>1-\epsilon_q$ and $g|_{E_q}$ is continuous. Then $m(\bigcup_{q=1}^{\infty}E_q)=1$.
    For each $q\in \mathbb{N}$, we conclude
	\begin{align*}
	\lim_{u\to\infty}|g_u(s,t)-g(s,t)|&\leq \lim_{u\to\infty}\bigg(2^{2u}\int_{\eta^u(s,t)\bigcap E_q}|g(\xi,\zeta)-g(s,t)|d\xi d\zeta\\
	&+2^{2u}\int_{\eta^u(s,t)\bigcap E_q^c}|g(\xi,\zeta)-g(s,t)|d\xi d\zeta\bigg)\\
	&\leq 0+\lim_{u\to\infty}2^{2u}\lk(m\lk(E_q^c \bigcap \eta^u(s,t)\re)\re)M\overset{\eqref{A}}=0,
	\end{align*}
	for $m$-a.e. $(s,t)\in E_q$, where $M=\max_{(s,t)\in [0,1)^2}\{|g(s,t)|\}$.
	It follows from $m(\bigcup_{q=1}^{\infty}E_q)=1$  that $g_u(s,t)$ converges to $g(s,t)$ for $m$-a.e. $(s,t)\in [0,1)^2$ as $u\to \infty$. 
\end{proof}

Now we are able to prove Claim \ref{c-1}.
\begin{proof}[Proof of Claim \ref{c-1}.]
	Before proving it, we introduce a notation, that is, $U_{\phi_{m,n}} h=h\circ\phi_{m,n}$ for all $h\in L^2(\widetilde{X},\widetilde{\mathcal{B}},\widetilde{\mu})$.
	Since $B\in  \mathcal{K_{\widetilde{\mu}}}$, it follows that
	$$\overline{\{W^n1_B:n\in\mathbb{Z}_+\}}\text{ is compact in }L^2(\widetilde{X},\widetilde{\mathcal{B}},\widetilde{\mu}),$$
	which implies that
	\begin{equation}\label{f-1}
	\overline{\{U_{\phi_{m,n}}1_B:|\beta m- n|<1,(m,n)\in \mathbb{Z}^2\}} \text{ is compact in } L^2(\widetilde{X},\widetilde{\mathcal{B}},\widetilde{\mu}).
	\end{equation}
	It follows from \eqref{f-1} that for any $\epsilon>0$ there exists a finite $\epsilon$-net $\{U_{\phi_{m_i,n_i}}1_B\}_{i=1}^p$ with $(m_i,n_i)\in \mb{Z}^2$ and $|\beta m_i-n_i|<1$. That is, for any $(m,n)\in \mathbb{Z}^2$ with $|\beta m-n|<1$, there exists $i\in \{1,2,\ldots,p\}$ such that $$\lk(\int|U_{\phi_{m_i,n_i}}1_B-U_{\phi_{m,n}}1_B|^2d(\mu\times m)\re)^{\frac{1}{2}}=\|U_{\phi_{m_i,n_i}}1_B-U_{\phi_{m,n}}1_B\|_{2}<\epsilon.$$ 
	For a fixed $u$, we consider   $\{U_{\phi_{m_i,n_i}}f_u\}_{i=1}^p$. By Minkowski's inequality for intergrals (see \cite[Page 194]{RA}), we conclude
	\begin{align*}	
	\|U_{\phi_{m_i,n_i}}f_u-U_{\phi_{m,n}}f_u&\|_{2}=\lk(\int_{X}\int_{[0,1)^2}|U_{\phi_{m_i,n_i}}f_u-U_{\phi_{m,n}}f_u|^2dm d\mu\re)^{\frac{1}{2}}\\
	&=\lk(\int_{X}\sum_{k,l}\int_{\eta^u_{k,l}}|U_{\phi_{m_i,n_i}}f_u-U_{\phi_{m,n}}f_u|^2dmd\mu\re)^{\frac{1}{2}}\\
	&=\lk(\int_{X}\sum_{k,l}\int_{\eta^u_{k,l}}|U_{T}^{(m_i,n_i)}\int_{\eta^u_{k,l}}1_Bdm-U_{T}^{(m,n)}\int_{\eta^u_{k,l}}1_Bdm|^2dmd\mu\re)^{\frac{1}{2}}\\
	&=\lk(\int_{X}\sum_{k,l}m(\eta^u_{k,l})|\int_{\eta^u_{k,l}}(U_{\phi_{m_i,n_i}}1_B-U_{\phi_{m,n}}1_B)dm|^2d\mu\re)^{\frac{1}{2}}\\
	&\le\lk(\int_{X}\sum_{k,l}m(\eta^u_{k,l})\int_{\eta^u_{k,l}}|U_{\phi_{m_i,n_i}}1_B-U_{\phi_{m,n}}1_B|^2dmd\mu\re)^{\frac{1}{2}}\\
	&=\frac{1}{2^u}\lk(\int_{X}\int_{[0,1)^2}|U_{\phi_{m_i,n_i}}1_B-U_{\phi_{m,n}}1_B|^2dmd\mu\re)^{\frac{1}{2}}\\
	&=\frac{1}{2^u}\lk(\int_{X\times [0,1)^2}|U_{\phi_{m_i,n_i}}1_B-U_{\phi_{m,n}}1_B|^2d(\mu\times m)\re)^{\frac{1}{2}}\\
	&<\frac{1}{2^u}\epsilon.
	\end{align*}
	Hence $\{U_{\phi_{m_i,n_i}}f_u\}_{i=1}^p$ is a finite $\epsilon/2^u$-net of $\overline{\{U_{\phi_{m,n}}f_u:|\beta m- n|<1,(m,n)\in \mathbb{Z}^2\}}$, which implies that $$\overline{\{U_{\phi_{m,n}}f_u:|\beta m-n|<1,(m,n)\in \mathbb{Z}^2\}}$$ is compact in  $L^2(\widetilde{X},\widetilde{\mathcal{B}},\widetilde{\mu})$ for all $u\in\mathbb{N}.$ For a fixed $u\in \mathbb{N}$, we take a finite $\epsilon/2^{u}$-net $\{U_{\phi_{m_i,n_i}}f_u\}_{i=1}^p.$ Note that
	$$U_{\phi_{m,n}}f_u(x,s,t)=f_u(T^{(m,n)}x,s,t).$$
	So for any $(m,n)\in \mathbb{Z}^2$ with $|\beta m_i-n_i|<1$, there exists $i\in\{1,2,\ldots,p\}$ such that
	\begin{equation}\label{A1}
	\begin{split}
	\frac{\epsilon}{2^{u}}&>\|U_{\phi_{m,n}}f_u-U_{\phi_{m_i,n_i}}f_u\|_{2}\\
&=\lk(\int_{X}\int_{[0,1)^2}|f_u(T^{(m,n)}x,s,t)-f_u(T^{(m_i,n_i)}x,s,t)|^2dmd\mu\re)^{\frac{1}{2}}.
	\end{split}
	\end{equation}
	For any $(s,t)\in [0,1)^2$, there exist $k,l\in\{0,\ldots, 2^u-1\}$ such that $(s,t)\in \eta_{k,l}^u$. Hence
	\begin{align*}
	&\|U_T^{(m,n)}f_u(\cdot,s,t)-U_T^{(m_i,n_i)}f_u(\cdot,s,t)\|_{2}\\
	=&\lk(\int_{X}|f_u(T^{(m,n)}x,s,t)-f_u(T^{(m_i,n_i)}x,s,t)|^2d\mu\re)^{\frac{1}{2}}\\
	=&\lk(\int_{X}2^{2u}\int_{\eta^u_{k,l}}|f_u(T^{(m,n)}x,s,t)-f_u(T^{(m_i,n_i)}x,s,t)|^2dmd\mu\re)^{\frac{1}{2}}\\
	\leq&2^u\lk(\int_{X}\int_{[0,1)^2}|f_u(T^{(m,n)}x,s,t)-f_u(T^{(m_i,n_i)}x,s,t)|^2dmd\mu\re)^{\frac{1}{2}}\\
	\overset{\eqref{A1}}\leq&2^{u}\cdot\frac{\epsilon}{2^{u}}=\epsilon.
	\end{align*}
	Hence for a fixed $(s,t)\in [0,1)^2$, $\overline{\lk\{U_{T}^{(m,m)}f_u:|\beta m- n|<1,(m,n)\in \mathbb{Z}^2\re\}}$ has  a finite $\epsilon$-net $\{U_{T}^{(m_i,n_i)}f_u(x,s,t)\}_{i=1}^p$. It follows that $$\overline{\{U_{T}^{(m,m)}f_u:|\beta m- n|<1,(m,n)\in \mathbb{Z}^2\}}$$  is compact in  $L^2(X,\mathcal{B}_X,\mu)$ for each $(s,t)\in [0,1)^2$. This ends the proof of the first statement in Claim \ref{c-1}.
	
	By Lemma \ref{lem-41}, we deduce that for $\mu$-a.e. $x\in X$
	$$\lim_{u\to\infty}f_u(x,s,t)=1_B(x,s,t)$$
	for $m$-a.e. $(s,t)\in[0,1)^2.$ By dominated convergence theorem, we conclude that
	$$\lim_{u\to\infty}\|f(\cdot,s,t)-1_{B^{(s,t)}}\|_{2}=0$$
	for $m$-a.e. $(s,t)\in[0,1)^2.$ Now we finish the proof of Claim \ref{c-1}.
\end{proof}
\section{Results for $\mathbb{Z}^q$-MPS}
In this section, we introduce the corresponding results of $\mathbb{Z}^q$-MPS, which are proved by exactly the same methods for the case of $\mathbb{Z}^2$-MPS.
Let $(X ,\mathcal{B}_X, \mu, T)$ be a $\mathbb{Z}^q$-MPS. For a fixed finite measurable partition $\alpha$ of $X$ and an infinite subset $S=\{\vec{w}_i\}_{i=1}^{\infty}$ of $\mathbb{Z}^q$, we put
$$ h^S_{\mu}(T,\alpha)= \limsup_{k\to \infty} \frac{1}{k}H_\mu \left(\bigvee_{i=1}^k T^{-\vec{w}_i} \alpha\right).$$
Then we can define the sequence entroy of $T$ for the infinite subset $S$  by
$$h^S_{\mu}(T)=\sup_{\alpha}h^S_{\mu}(T,\alpha),$$
where the supremum is taken over all finite measurable partitions of $X$. 

Now we define directional sequence entropy for the case of $\mb{Z}^q$-actions. Let $\vec{v}=(1,\beta_2,\ldots,\beta_q)\in\mathbb{R}^q$ be a direction vector and $\textbf{b}=(b_2,\ldots,b_q)\in \mathbb{R}^{q-1}_+:=\{\textbf{u}=(u_1,\ldots,u_{q-1})\in \mathbb{R}^{q-1}:u_i>0\}$.
 We put $$\Lambda^{\vec{v}}(\textbf{b})=\{\vec{w}=(m_1,\ldots,m_q)\in\mathbb{Z}^q:\beta_i m_1-b_i/2\leq m_i\leq \beta_i m_1+b_i/2,\quad i\in\{2,\ldots,q\}\}.$$ 
Given a fixed finite measurable partition $\alpha$ of $X$ and an infinite subset $S=\{\vec{w}_i=(m_1^{(i)},\ldots,m_q^{(i)})\}_{i=1}^{\infty}$ of $\Lambda^{\vec{v}}(\textbf{b})$ that $\{m_1^{(i)}\}_{i=1}^{\infty}$ is strictly monotone, we put
$$ h^S_{\mu}(T,\alpha)= \limsup_{k\to \infty} \frac{1}{k}H_\mu \lk(\bigvee_{i=1}^k T^{-\vec{w}_i} \alpha\re).$$
Then we can define the directional sequence entroy of $T$ for the infinite subset $S$  by
$$h^S_{\mu}(T)=\sup_{\alpha}h^S_{\mu}(T,\alpha),$$
where the supremum is taken over all finite measurable partitions of $X$.    	
 Define the $\vec{v}$-directional Kronecker algebra by
$$\mathcal{K}_\mu^{\vec{v}}(\textbf{b})=\lk\{B\in\mathcal{B}_X: \overline{\{U_T^{\vec{w}}1_B :\vec{w}\in \Lambda^{\vec{v}}(\textbf{b}) \}}\text{ is compact in } L^2(X,\mathcal{B}_X,\mu) \re\},$$
where $U_T^{\vec{w}}:L^2(X,\mathcal{B}_X,\mu)\to L^2(X,\mathcal{B}_X,\mu)$ is the unitary operator such that
$$U_T^{\vec{w}}f=f\circ T^{\vec{w}}\text{ for all }f\in L^2(X,\mathcal{B}_X,\mu).$$
Similarly, we can prove that $\mathcal{K}_\mu^{\vec{v}}(\textbf{b})$ is a $\sigma$-algebra and the definition of $\mathcal{K}_\mu^{\vec{v}}(\textbf{b})$ is independent of the selection of $\textbf{b}$. So we omit $\textbf{b}$ in $\mathcal{K}_\mu^{\vec{v}}(\textbf{b})$ and write it as $\mathcal{K}_\mu^{\vec{v}}$.
\medskip
\begin{thm} \label{thm-B3}
	Let $(X,\mathcal{B}_X,\mu,T)$ be a  $\mathbb{Z}^q$-MPS,  $\vec{v}=(1,\beta_2,\ldots,\beta_q)\in\mathbb{R}^q$ be a direction vector and  $\textbf{b}=(b_2,\ldots,b_q)\in \mathbb{R}^{q-1}_+$. Given a finite measurable partition $\alpha$ of $X$. For any infinite subset $S'$ of $\Lambda^{\vec{v}}(\textbf{b})$, $$h^{S'}_{\mu}(T,\alpha)\leq H_\mu(\alpha|\mathcal{K}_\mu^{\vec{v}}).$$ Moreover, there exists an infinite subset $S=\{\vec{w}_i=(m_1^{(i)},\ldots,m_q^{(i)})\}_{i=1}^{\infty}$ of $\Lambda^{\vec{v}}(\textbf{b})$ such that $\{m_1^{(i)}\}_{i=1}^{\infty}$ is strictly monotone and $$h^{S}_{\mu}(T,\alpha)=H_\mu(\alpha|\mathcal{K}_\mu^{\vec{v}}).$$
\end{thm}
\begin{rem}Let $\alpha$ be a finite measurable partition of $X$. We 
	define \begin{align}\label{B4}h_\mu^{\vec{v},*}(\alpha)=\max\limits_{S\subset \Lambda^{\vec{v}}(\textbf{b})}\{h^{S}_{\mu}(T,\alpha)\}.\end{align} By Theorem \ref{thm-B3}, $h_\mu^{\vec{v},*}(\alpha)=H_\mu(\alpha|\mathcal{K}_\mu^{\vec{v}})$ and then
	\eqref{B4} is well defined since it is independent of the selection of $\textbf{b}$.
\end{rem}
Corresponding to the case of $\mathbb{Z}^2$-MPS, we have the following definitions.
\begin{defn}
	Let $(X,\mathcal{B}_X,\mu,T)$ be a $\mathbb{Z}^q$-MPS and $\vec{v}=(1,\beta_2,\ldots,\beta_q)\in\mathbb{R}^q$ be a direction vector.
	\begin{itemize}
		\item[(a)]We say $\mu$ has $\vec{v}$-discrete spectrum if $\mathcal{K}_\mu^{\vec{v}}=\mathcal{B}_X$.
	
		\item[(b)] We say $(X,\mathcal{B}_X,\mu,T)$ is  $\vec{v}$-null if for any  $\textbf{b}=(b_2,\ldots,b_q)\in \mathbb{R}_+^{q-1}$ and infinite subset $S=\{\vec{w}_i=(m_1^{(i)},\ldots,m_q^{(i)})\}_{i=1}^{\infty}$ of $\Lambda^{\vec{v}}(\textbf{b}) $, $h^{S}_{\mu}(T)=0$.
	
	\end{itemize}
	
\end{defn}
\medskip

Now we state main conclusions for the case of $\mathbb{Z}^q$-actions corresponding to $\mathbb{Z}^2$-actions,  which are proved by exactly the same methods for the case of $\mathbb{Z}^2$-MPS as follows. 
\begin{thm}\label{thm-B4}
		Let $(X,\mathcal{B}_X,\mu,T)$ be a $\mathbb{Z}^q$-MPS and $\vec{v}=(1,\beta_2,\ldots,\beta_q)\in\mathbb{R}^q$ be a direction vector.  Then the following two conditions are equivalent.

		\begin{itemize}
		\item[(a)] $\mu$ has $\vec{v}$-discrete spectrum system.
		\item[(b)]$(X,\mathcal{B}_X,\mu,T)$ is $\vec{v}$-null.
	\end{itemize}
\end{thm}

\begin{thm}	Let $(X,\mathcal{B}_X,\mu,T)$ be a $\mathbb{Z}^q$-MPS.
	Given a finite measurable partition $\alpha$ of $X$. If there exists an infinite subset $S$ of $\mathbb{Z}^q$, such that $h^S_{\mu}(T,\alpha)>0$, then there are at most $q-1$ linearly independent directions $\{\vec{v_i}=(1,\beta^i_2,\ldots,\beta^i_q)\}_{i=1}^{q-1}\subset\mathbb{R}^q$ such that $$h_\mu^{\vec{v_i},*}(\alpha)=0, \text{ } i=1,2,\ldots,q-1.$$
\end{thm}

\begin{thm}	
	Let $(X,\mathcal{B}_X,\mu,T)$ be a $\mathbb{Z}^q$-MPS. Then the following statements are equivalent.
	\begin{itemize}
		\item[(a)] There exist $q$ linearly independent directions $\{\vec{v_i}=(1,\beta^i_2,\ldots,\beta^i_q)\}_{i=1}^{q}\subset\mathbb{R}^q$ such that $\mu$ has $\vec{v_i}$-discrete spectrum for all $i=1,\ldots,q$.
		\item[(b)]$\mu$ has discrete spectrum.
	\end{itemize}
\end{thm}

\end{appendix}

\end{document}